\documentclass[12pt]{article}

\input{style}

\usepackage{amsmath, amssymb, amstext, amsfonts, mathrsfs, bbold}
\usepackage{stmaryrd, url,paralist}
\usepackage{mathtools}

\mathtoolsset{showonlyrefs}

\usepackage{geometry}
\newtheoremstyle{results}
  {1ex}
  {1ex}
  {\itshape}
  {}
  {\bfseries}
  {.}
  {.5em}
  {}

\theoremstyle{results}

	\newtheorem{Sat}{Satz}[section]
	\newtheorem{Thm}[Sat]{Theorem}
	\newtheorem{DefThm}[Sat]{Definition and Theorem}
	\newtheorem{Lem}[Sat]{Lemma}
	\newtheorem{Pro}[Sat]{Proposition}
	\newtheorem{Cor}[Sat]{Corollary}

	\newtheorem{Def}[Sat]{Definition}
	\newtheorem{Ass}[Sat]{Assumption}	
	
\newtheoremstyle{remarks}
  {1ex}
  {1ex}
  {}
  {}
  {\itshape}
  {.}
  {.5em}
  {}
  
\theoremstyle{remarks}
	\newtheorem{Ex}[Sat]{Example}

	\newtheorem{Rem}[Sat]{Remark}

	\newcommand{\FA}{\mathcal{F}}
	\newcommand{\aFA}{\mathcal{F}^\Lambda}
	\newcommand{\WM}{\mathbb{P}}
	\newcommand{\NZ}{\mathbb{N}}
	
	\newcommand{\stm}{{\mathcal{S}^\Lambda}}
	\newcommand{\stmd}{{{\mathcal{S}^{\Lambda,\mathrm{div}}}}}	
	
	\newcommand{\stmg}[1]{\stm\left([#1,\infty]\right)}
	
	\newcommand{\st}{{\mathcal{S}^\mathcal{O}}}
	\newcommand{\stp}{{\mathcal{S}^\mathcal{P}}}
	
	\newcommand{\midG}{\,\middle|\,}
	
	\newcommand{\stsetRO}[2]{\left\llbracket #1,#2 \right\llbracket}
	
	\newcommand{\stsetO}[2]{\left\rrbracket #1,#2 \right\llbracket}
	\newcommand{\stsetG}[1]{\left\llbracket #1 \right\rrbracket}
	
	\newcommand{\lsl}[1]{{^\ast #1}}
	\newcommand{\lsr}[2]{{#1^\ast_{#2}}}
	
	\newcommand{\citing}[3]{\cite{#1}, #2, p.#3}

\begin{document}

\makeatletter
\let\@fnsymbol\@arabic
\makeatother

\title{\vspace*{-3cm} On Lenglart's Theory of Meyer-$\sigma$-fields and El Karoui's Theory of Optimal Stopping}

\author{ Peter Bank\footnote{Technische Universit{\"a}t Berlin,
    Institut f{\"u}r Mathematik, Stra{\ss}e des 17. Juni 136, 10623
    Berlin, Germany, email \texttt{bank@math.tu-berlin.de}. }
  \hspace{4ex} David Besslich\footnote{Technische Universit{\"a}t
    Berlin, Institut f{\"u}r Mathematik, Stra{\ss}e des 17. Juni 136,
    10623 Berlin, Germany, email \texttt{besslich@math.tu-berlin.de}.}  }
\date{\today}

\maketitle
\begin{abstract}
	We summarize the general results
	of \cite{EK81} on optimal stopping problems for 
	processes which are measurable with respect to Meyer-$\sigma$-fields.
	Meyer-$\sigma$-fields are due to \cite{EL80} and include the optional and predictable $\sigma$-field as
	special cases. Novel contributions of our work
	are path regularity results for Meyer measurable processes and limit results for 
	Meyer-projections. We also clarify a minor issue in the proof of El Karoui's optimality result.
	These extensions were inspired and needed for the proof of a stochastic representation theorem in~\cite{BB18}.
	As an application of this theorem, we provide an alternative approach to optimal stopping in the spirit of \cite{BF03}.
\end{abstract}

\begin{description}
\item[Mathematical Subject Classification (2010):] 60G40, 60G07
\item[Keywords:] Optimal stopping problems, Snell envelope, 
Meyer-$\sigma$-fields
\end{description}

\tableofcontents

\section{Introduction}

In recent work on stochastic optimal control problems (\cite{BB18_3}), we found it useful to model information flows that allow one to interpolate between predictable and optional controls. In fact, such a situation is perfectly natural when exogenous shocks are pre-announced by possibly noisy signals. For continuous-time models, Meyer-$\sigma$-fields turn out to be the tool of choice to capture such phenomena in a most flexible and rigorous manner. 

Meyer $\sigma$-fields where introduced by \cite{EL80} and used in \cite{EK81} for a most general theory of optimal stopping. This paper gives a survey of both of these papers, giving a condensed account of the main concepts and results we find most important for our optimal control purposes. Moreover, it provides supplementary results to both Lenglart's and El Karoui's theory which shed extra light on the regularity properties of Meyer-measurable processes and allow us to clarify minor issues in the proof of El Karoui's optimal stopping result. 

From Lenglart's theory, we recall the definition of Meyer-$\sigma$-fields, discuss their completion and how they are embedded between predictable and optional fields. Also, accessible and totally inaccessible Meyer-stopping times are discussed as are the fundamental tools provided by Meyer's section and projection theorems. 

Our survey of El Karoui's theory of optimal stopping starts with a discussion of Meyer-supermartingales and the structure of Snell envelopes. We explain how she relaxes the optimal stopping problem by introducing divided stopping times and show how this relaxation ensures existence of optimizers. To round off this survey, we recall how the results simplify in the classically considered optional case.

Apart from this survey section, we will clarify a minor issue arising in the analysis in \cite{EK81}. Moreover, we give extensions
to existing results concerning the testing of path properties for Meyer-measurable
processes and a result concerning the right- and left-upper-semicontinuous
envelopes of Meyer-projections. Those extensions are not just for the sake
of mathematical generality, but were inspired and needed in our paper \cite{BB18},
where we prove some extension of a representation result for stochastic processes from \cite{BK04}. This result in turn is used in \cite{BB18_3} to construct solutions to irreversible investment problems for which Meyer $\sigma$-fields offer a novel way to model information dynamics. 

For another application, we show how the solution to the representation theorem of \cite{BB18} can be used as a universal stopping signal which allows one to characterize optimal divided stopping times for a suitable parametric family of optimal stopping problems.

The rest of the paper is organized as follows. In Section \ref{sec:meyer} 
we follow \cite{EK81} to give a summary of the key results on Meyer-$\sigma$-fields 
from \cite{EL80}. In Section 
\ref{sec:optstop} we state the results of \cite{EK81}
 concerning optimal stopping problems
and in Section \ref{sec:extension}
 we state our extensions and clarifications on \cite{EK81}. Section \ref{sec:rep} studies the universal signals for optimal stopping as specified by a solution to a representation problem.

\section{Survey of results on Meyer-$\sigma$-fields and general optimal stopping}

This section gives a condensed account of the results from Lenglart's general theory of Meyer-$\sigma$-fields and El Karoui's general theory of optimal stopping that we found most useful for our own work in the companion papers \cite{BB18, BB18_3}. 

\subsection{Lenglart's theory of Meyer-$\sigma$-fields}\label{sec:meyer}

	\cite{EL80}
	gives a very thorough and comprehensive account of the theory of Meyer- $\sigma$-fields. Let us recall some of his results by following the outline given in the introduction of \cite{EK81}, p.118-121.  
	
	\subsubsection{Basic definition, examples and characterization result}
	
	We will start with the definition of Meyer-$\sigma$-fields and
	some fundamental examples.
	
	\begin{Def}[\citing{EL80}{Definition 2}{502}]
			\label{meyer_def_1}
				A $\sigma$-field $\Lambda$ on $\Omega\times [0,\infty)$
				is called a \emph{Meyer-$\sigma$-field},
				if the following conditions hold:
				\begin{enumerate}[(i)]
					\item It is generated by some right-continuous, left-limited
					(rcll or c\`adl\`ag in short) processes.
					\item It contains 
					$\{\emptyset,\Omega\}\times \mathcal{B}([0,\infty))$,
					where $\mathcal{B}([0,\infty))$
					denotes the Borel-$\sigma$-field on $[0,\infty)$.
					\item It is stable with respect to stopping at deterministic
					time points, i.e. for a $\Lambda$-measurable process $Z$,
					$s\in [0,\infty)$, also the stopped process 
					$(\omega,t)\mapsto Z_{t\wedge s}(\omega)$ is 
					$\Lambda$-measurable.
				\end{enumerate}
	\end{Def}	

	\begin{Ex}[\citing{EK81}{Remark}{p.118}]\label{meyer_ex_1}
			Assume we are given a filtered  probability space  	
	$(\Omega,\mathbb{F}, (\mathcal{F}_t)_{t\geq 0},
	\mathbb{P})$. Then
			 the optional $\sigma$-field with respect to the filtration
			$(\mathcal{F}_t)_{t\geq 0}$, 
			 i.e. the $\sigma$-field generated by all c\`adl\`ag, 
			$(\mathcal{F}_t)_{t\geq 0}$-adapted processes and the
			predictable $\sigma$-field with respect to the filtration
			$(\mathcal{F}_t)_{t\geq 0}$, 
			i.e. the $\sigma$-field generated by all continuous, 
			$(\mathcal{F}_t)_{t\geq 0}$-adapted processes are both
			Meyer-$\sigma$-fields.
	\end{Ex}
	
	So Meyer-$\sigma$-fields can emerge from suitable processes adapted to a given filtration. Conversely, Meyer-$\sigma$-fields also induce a filtration:
	
	\begin{Def}[Compare \cite{EK81}, p.119]\label{meyer_def_2}
			For a Meyer-$\sigma$-field $\Lambda$ we define
			its \emph{associated
			filtration} $\mathcal{F}^\Lambda:=
			(\mathcal{F}^\Lambda_t)_{t\geq 0}$ by setting 
			\[
			    \mathcal{F}^\Lambda_t:=\sigma\left(Z_t\,\middle|\, Z \text{ $\Lambda$-measurable}\right),\quad t\in [0,\infty).
			\]
			In addition, one can choose a $\sigma$-field $\aFA_{0-}\subset \aFA_0$ and we put
			$\mathcal{F}^\Lambda_\infty
			:=\bigvee_{t=1}^\infty \mathcal{F}^\Lambda_t$ as well as
			\[
				\mathcal{F}^\Lambda_{t+}:=\bigcap_{s>t}
				\mathcal{F}^\Lambda_s,\ t\geq 0,\quad (\mathcal{F}^\Lambda_+)_{0-}
				:=\mathcal{F}^\Lambda_{0-}.
			\]
			A process $Z:\Omega\times [0,\infty]\rightarrow 
				\mathbb{R}$ is called $\Lambda$-measurable, 
				if $Z_\infty$ is $\mathcal{F}^\Lambda_\infty$-measurable 
				and the restriction $Z|_{\Omega\times [0,\infty)}$
				 is $\Lambda$-measurable. 
	\end{Def}
	
	The next theorem gives us some idea what Meyer-$\sigma$-fields
	look like. 
	
	\begin{Thm}[Compare \citing{EK81}{Definition 2.22.2}{119}]
			\label{meyer_thm_1}
			A Meyer-$\sigma$-field contains the predictable
			$\sigma$-field $\mathcal{P}(\mathcal{F}^\Lambda)$
			relative to the filtration $(\mathcal{F}^\Lambda_t)_{t\geq 0}$ 
			and it is contained in the optional $\sigma$-field
			$\mathcal{O}(\mathcal{F}^\Lambda)$ relative to
			$(\mathcal{F}^\Lambda_t)_{t\geq 0}$.
			
			Conversely,  a $\sigma$-field on  
			$\Omega\times [0,\infty)$ generated by
			c\`adl\`ag processes is a Meyer-$\sigma$-field, if
			 it lies between the predictable
			and the optional $\sigma$-field of some filtration.
	\end{Thm}	
	
 \subsubsection{Stopping times corresponding to a Meyer-$\sigma$-field
 $\Lambda$} Next we will give a definition for the concept of
stopping times when using general Meyer-$\sigma$-fields.
	
	\begin{Def}[\citing{EK81}{Definition 2.22.2}{119}]
			\label{meyer_def_3}	
			A random variable $S$ with values in $[0,\infty]$
			is a \emph{$\Lambda$-stopping time},
			if 
			\[
				[[S,\infty[[:=\left\{(\omega,t)\in \Omega\times 
				[0,\infty)\, \middle | \, S(\omega)\leq t\right\} \in \Lambda.
			\]			
			The set of all $\Lambda$-stopping times is denoted
			by $\stm$.
			Additionally we define for each mapping 
			$S:\Omega\rightarrow [0,\infty]$ a
			$\sigma$-field 
			\[
			\mathcal{F}^\Lambda_S := \sigma(Z_S \,|\, Z \text{ $\Lambda$-measurable process})
			\]				
	\end{Def}
	
	This concept of $\Lambda$-stopping times extends classical notions of stopping times in a natural way:
	
	\begin{Ex}[Compare \cite{EK81}, p.119]\label{meyer_ex_2}
			For a filtration $(\mathcal{F}_t)_{t\geq 0}$ and 
			$\Lambda=\mathcal{O}(\mathcal{F})$, a $\Lambda$-stopping
			time $S$ is a
			 classical stopping time associated to the filtration
			$(\mathcal{F}_t)_{t\geq 0}$
			 and $\mathcal{F}^\Lambda_S=\mathcal{F}_S$.
			 For $\Lambda=\mathcal{P}(\mathcal{F})$ by contrast a
			$\Lambda$-stopping time $S$ is a
				predictable stopping time associated to the filtration
				$(\mathcal{F}_t)_{t\geq 0}$. 
	\end{Ex}
	
	\begin{Rem}[Compare \cite{EK81}, p.119]\label{meyer_rem_1}
			The $\sigma$-fields from Definition \ref{meyer_def_3}
			satisfy 
			\[
				\mathcal{F}^\Lambda_{S-}\subset \mathcal{F}^\Lambda_S
				\subset \mathcal{F}^\Lambda_{S+},
			\]
			where $\mathcal{F}^\Lambda_{S+}:=(\mathcal{F}^\Lambda_+)_S$
			with $\mathcal{F}^\Lambda_+$ the right-continuous
			filtration defined in Definition \ref{meyer_def_2} and $\mathcal{F}^\Lambda_{S-}:=
			({ \mathcal{F}^\Lambda})_S^{\mathcal{P}({ \mathcal{F}^\Lambda})}$,
			which is in the case of an $\mathcal{F}^\Lambda_+$-stopping time
			$S$ equal to
			the $\sigma$-field
			generated by $\mathcal{F}_{0-}^\Lambda$ and the sets $A\cap
			\{t<S\}$, $t\geq 0$, $A
			\in \mathcal{F}_t^\Lambda$ (see \citing{EL80}{Remark}{505}).
			By \citing{EL80}{Theorem 4.1}{505}, we have for a 
			$\Lambda$-stopping time $S$ 
			\[
				\mathcal{F}^\Lambda_S=\left\{H\in 
					\mathcal{F}^\Lambda_{\infty} \ \middle\vert\  
					S_H \in \stm \right\},
			\]
			where $S_H$ is defined to be
			\[
			    S_H:=\begin{cases}
			            S&\text{ on } H,\\
			            \infty &\text{ on } H^c.
			        \end{cases}
			\]
	\end{Rem}
		
	\subsubsection{Meyer's section and projection theorems}	
		
	Next we state the key section and projection theorems,
	which are well known
	for the optional and the predictable $\sigma$-field, but 
	actually still
	hold for Meyer-$\sigma$-fields. We will fix for the rest of this
	section a complete probability space
	$(\Omega,\mathbb{F},\mathbb{P})$ and a Meyer-$\sigma$-field 
	$\Lambda$ (not necessarily complete; see Definition \ref{meyer_def_4} below), which is contained
	in $\mathbb{F}\otimes \mathcal{B}([0,\infty))$.
	
	\begin{Thm}[Meyer Section Theorem, \citing{EK81}{Theorem 2.23.1}{120}]\label{meyer_thm_2}
			Let $B$ be an element of $\Lambda$. For every
				$\epsilon>0$, there exists $S\in \stm$
				such that $B$ contains the graph of $S$, i.e. 
				\[ 
					B\supset \mathrm{graph}(S):=\stsetG{S}:=\{(\omega,S(\omega))\in \Omega\times 
				[0,\infty)	\, | \, S(\omega)<\infty\}
				\]								
				and
				\[
					\mathbb{P}(S<\infty)>\mathbb{P}(\pi(B))-\epsilon,
				\]
				where $\pi(B):=\left\{\omega\in \Omega \,
				\middle |\, (\omega,t)\in B
					 \text{ for some $t\in [0,\infty)$}\right\}$
				denotes the projection of $B$ onto $\Omega$.
	\end{Thm}
	
	\begin{Rem}
		The projection $\pi(B)$ of a set $B\in \Lambda$ is an element of $\mathbb{F}$ 
		as the probability space is assumed to be complete. In general we would have to
		replace $\mathbb{P}(\pi(B))$ by $\mathbb{P}^\ast(\pi(B))$, where 
		$\mathbb{P}^\ast$ denotes the outer measure
		of $\mathbb{P}$ (see \citing{DM78}{Footnote (1)}{137}).
	\end{Rem}
	
	An important consequence of Theorem \ref{meyer_thm_2} is the following corollary:
	
	\begin{Cor}[\cite{EL80}, p. 507]	\label{meyer_cor_1}
		If $Z$ and $Z'$ are two $\Lambda$-measurable processes,
		such that for each bounded $T\in \stm$ we have
		$Z_T\leq Z_T'$ a.s. (resp. $Z_T=Z_T'$ a.s.), then the 
		set $\{Z>Z'\}$ is evanescent (resp. $Z$ and $Z'$ are
		indistinguishable).
	\end{Cor}
	
	Next it is also possible to project a suitable process 
	into the space of $\Lambda$-measurable processes:
	
	\begin{Thm}[Projection Theorem, compare \citing{EL80}{Theorem 8}{511}]\label{meyer_thm_3}
				For any bounded or positive $\mathbb{F}\otimes
				\mathcal{B}([0,\infty))$-measurable process $Z$,  
				there exists a 
				$\Lambda$-measurable process $^\Lambda Z$, unique up 
				to indistinguishability, such that 
				\[
					{^\Lambda Z}_S=\mathbb{E}\left[Z_S|
					\mathcal{F}^\Lambda_S\right]\quad \mathbb{P}
					\text{-a.s. for any finite $S\in \stm$}.			
				\]
				This process is called \emph{$\Lambda$-projection of $Z$}.
	\end{Thm}
	
	\begin{Ex}\label{meyer_ex_3}
		If $\Lambda$ is the optional or predictable $\sigma$-field
		with respect to a filtration $(\mathcal{F}_t)_{t\geq 0}$
		then $^\Lambda Z$ coincides with the well known optional and 
		predictable projection, respectively.
	\end{Ex}
	
	\begin{Rem}\label{meyer_rem_2}
			By \citing{EL80}{Theorem 11}{513}, we can use Theorem
			\ref{meyer_thm_3} also for processes $Z$
			of class($D^\Lambda$), i.e. when $\{Z_T\ | \ T\in
			\stm\}$
			is uniformly integrable.
	\end{Rem}
	
	The Meyer Section Theorem and the definition of $\Lambda$-stopping times yield the following equivalent characterization:
	
	\begin{Thm}
			For any $\mathbb{F}\otimes
				\mathcal{B}([0,\infty))$-measurable process $Z\geq 0$,  
			the $\Lambda$-projection $^\Lambda Z$ is the, unique up 
			to indistinguishability, process satisfying 
			\[
				\mathbb{E}\left[\int_{[0,\infty)} Z_s\mathrm{d} A_s\right]
				=\mathbb{E}\left[\int_{[0,\infty)} {^\Lambda Z}_s\mathrm{d} A_s\right]
			\]
			for any c\`adl\`ag, $\Lambda$-measurable, increasing process
			$A$.
	\end{Thm}

	\subsubsection{Completion of Meyer-$\sigma$-fields} 
	
		Let us next analyze  
		the influence of $\mathbb{P}$-nullsets on $\Lambda$-stopping times and
		$\Lambda$-measurable process. Recall that we
		have fixed a probability space
		$(\Omega,\mathbb{F},\mathbb{P})$ and a Meyer-$\sigma$-field 
		$\Lambda$, which is contained
		in $\mathbb{F}\otimes \mathcal{B}([0,\infty))$.
	
		\begin{Def}[Compare \citing{EL80}{Definition and Theorem 2}{507}]
		\label{meyer_def_4}
			A Meyer-$\sigma$-field $\Lambda$ is called $\mathbb{P}$-complete if
			and only if one of the following equivalent conditions is fulfilled:
			\begin{enumerate}[(i)]
				\item Every mapping $T:\Omega\rightarrow 
				[0,\infty]$, which is almost surely
				equal to a $\Lambda$-stopping time is a $\Lambda$-stopping time.
				\item Every $\mathbb{F}\otimes
				\mathcal{B}([0,\infty))$-measurable process, which is 
				indistinguishable from a $\Lambda$-measurable
				process is itself $\Lambda$-measurable.
			\end{enumerate}
		\end{Def}
		
		The next statement shows how to obtain a completion of a Meyer-$\sigma$-field:
		Meyer-$\sigma$-field $\Lambda$:
		
		\begin{DefThm}[Compare \cite{EL80}, p.507-509]
				Define $\Lambda^{\mathbb{P}}$ as the $\sigma$-field generated by
				the stochastic intervals $\stsetRO{T}{\infty}$ for 
				random variables $T:\Omega\rightarrow [0,\infty]$
				which almost surely coincide with $\Lambda$-stopping time. Then
				the following results hold true:
				\begin{enumerate}[(i)]
					\item $\Lambda^{\mathbb{P}}$ is the smallest 
								$\mathbb{P}$-complete Meyer-$\sigma$-field containing $\Lambda$.	
					\item 	 
					The mapping $T:\Omega\rightarrow
			[0,\infty]$
			is a $\Lambda^{\mathbb{P}}$-stopping time if and
			only if the graph of $T$ is
			contained in $\Lambda^{\mathbb{P}}$.
					\item A random variable $T:\Omega\rightarrow [0,\infty]$
							is a $\Lambda^{\mathbb{P}}$-stopping time if and only if
							it is a.s. equal to a $\Lambda$-stopping time. 
					\item Fix a $\Lambda^{\mathbb{P}}$-stopping time $T$
					and take any corresponding $\Lambda$-stopping time 
					$\tilde{T}$ with $T=\tilde{T}$
					almost surely. Then
					$\mathcal{F}^{\Lambda^{\mathbb{P}}}_T=
						\bar{\mathcal{F}}^{\Lambda}_{\tilde{T}}$, where
						$\bar{\mathcal{F}}^{\Lambda}_{\tilde{T}}$ denotes
						the $\sigma$-field generated by 
						$\mathcal{F}^{\Lambda}_{\tilde{T}}$ and all 
						$\mathbb{P}$-nulsets.
				\end{enumerate}
				We call $\Lambda^{\mathbb{P}}$ the $\mathbb{P}$-completion of
				$\Lambda$.
		\end{DefThm}
			
		Analogously to Theorem~\ref{meyer_thm_1}, the following 
		theorem characterizes the $\mathbb{P}$-complete Meyer-$\sigma$-fields:
		
		\begin{Thm}[\citing{EL80}{Theorem 5}{509}]\label{meyer_thm_4}
			A $\sigma$-field $\Lambda$ generated by c\`adl\`ag processes is a 
			$\mathbb{P}$-complete Meyer-$\sigma$-field if and 
			only if $\Lambda$ is in between the 
			predictable and optional $\sigma$-field of a filtration,
			which is right-continuous and
			$\mathbb{P}$-complete, i.e. which fulfills the usual conditions.
		\end{Thm}
		
		\subsubsection{Inaccessible stopping times and connection to Meyer-measurable processes}	
	
		This paragraph follows \cite{EL80}, Section 3, p.510, and \cite{EL80}, Section 4, p.513. We give a small extension to those results in Proposition \ref{Main:778}.
		
		\begin{Def}[\citing{EL80}{Definition}{510}]\label{meyer_acc_def}
			For a Meyer-$\sigma$-field $\Lambda$ a random variable 
			$T:\Omega\rightarrow [0,\infty]$ is called \emph{$\Lambda$-accessible} if there exists
			a sequence of $\Lambda$-stopping times $(T_n)_{n\in \mathbb{N}}$ such that 
			\[
				\mathbb{P}\left(\bigcup_{n\in \NZ} \left\{T_n=T<\infty\right\}\right)=\mathbb{P}(T<\infty).
			\]
			The random variable $T$ is called \emph{totally $\Lambda$-inaccessible} if $\mathbb{P}(S=T<\infty)=0$ for all $\Lambda$-stopping times $S$.
		\end{Def}
		
		Now we have the following decomposition result:
		
		\begin{Thm}[\citing{EL80}{Theorem 6}{510}]\label{meyer_acc_Thm}
			Let $\Lambda$ be a Meyer-$\sigma$-field contained in another Meyer-$\sigma$-field $\bar{\Lambda}$. If $T$ is a $\bar{\Lambda}$-stopping time there exists a partition of $\{T<\infty\}$ into $A, I\in \FA_T^{\bar{\Lambda}}$ such that $T_A$ is
			a $\bar{\Lambda}$-stopping time $\Lambda$-accessible and $T_I$ is
			a $\bar{\Lambda}$-stopping time totally $\Lambda$-inaccessible. The sets $A,I$ are unique up to $\WM$-nulsets.
		\end{Thm}
		
		The next result characterizes $\Lambda$-measurability of a process via its sections at certain stopping times.
		
		\begin{Thm}[Compare \citing{EL80}{Theorem 13}{513}]\label{meyer_acc_Thm_2}
			Let $\Lambda$ be $\mathbb{P}$-complete and denote by $\FA:=(\FA_t)_{t\geq 0}$ a filtration satisfying the usual conditions such that $\mathcal{P}(\FA)\subset \Lambda\subset \mathcal{O}(\FA)$ (see Theorem \ref{meyer_thm_4}).
			
			A c\`adl\`ag process $X$ is $\Lambda$-measurable if and only if it satisfies the following two conditions:
			\begin{enumerate}
				\item[(i)] For any finite $\Lambda$-stopping time $T$, $X_T$ is $\mathcal{F}^\Lambda_T$-measurable.
				\item[(ii)] For any totally $\Lambda$-inaccessible $\FA$-stopping time $T$, $\Delta X_T=0$ a.s.. 
			\end{enumerate}
		\end{Thm}
		
		As a generalization of (ii) in the above result we note the following observation:
	
		\begin{Pro}\label{Main:778}
			In the setting of Theorem \ref{meyer_acc_Thm_2}, we have $\WM$-almost surely for any $\Lambda$-measurable bounded process $C$
					\[
				C_T=(^\mathcal{P} C)_T
			\]
			at any totally $\Lambda$-inaccessible $\mathcal{F}$-stopping time $T$.
			\end{Pro}			
			
			\begin{proof}
				The result follows by a monotone class argument, if we can show the result for 
				$C:=\alpha\mathbb{1}_{\stsetRO{S}{\infty}}$ for $S\in \stm$ and positive, bounded $\alpha \in \aFA_S$. In this situation, we can find by Theorem \ref{meyer_acc_Thm} two disjoint sets $A$, $I \in \aFA_S$ such that $\{S<\infty\}=A\cup I$, $S_I$ is totally $\mathcal{P}$-inaccessible and $S_A$ is $\mathcal{P}$-accessible. Hence, there exists a sequence of $\FA$-predictable stopping times $(S_n)_{n\in \NZ}$ such that 
				\[
					\mathbb{P}(\cup_{n\in \NZ}\{S_n=S_A<\infty\})=\WM(S_A<\infty).
				\]
				We assume without loss of generality that the graphs of the $(S_n)_{n\in \NZ}$ are disjoint, because otherwise
				we could consider the sequence of $\FA$-predictable stopping times 
				\[
					\tilde{S}_1:=S_1,\quad \tilde{S}_n:=\left(S_n\right)_{\cup_{k=1}^{n-1}\{S_n\neq S_k\}},\quad n=2,3,\dots.
				\]
				Now we have
				\begin{align}\label{Main:780}
					{^\mathcal{P}C}={^\mathcal{P} \left(\alpha \mathbb{1}_{\stsetG{S_A}}\right)}
					+{^\mathcal{P} \left(\alpha \mathbb{1}_{\stsetG{S_I}}\right)}
					+\alpha \mathbb{1}_{\stsetO{S}{\infty}}.
				\end{align}
				As $S_I$ is totally $\mathcal{P}$-inaccessible we get at any predictable stopping time $T$ 
				\[
				    \left(\alpha\mathbb{1}_{\stsetG{S_I}}\right)_T=\alpha\mathbb{1}_{\{S_I=T\}}=0\quad \text{a.s.,}
				\]
				and so, by Corollary \ref{meyer_cor_1} to Meyer's Section Theorem:
				\begin{align}\label{Main:781}
					{^\mathcal{P} \left(\alpha \mathbb{1}_{\stsetG{S_I}}\right)} \equiv 0.
				\end{align}
				Moreover, by monotonicity and additivity of the $\mathcal{P}$-projection
				 and because the graphs of the $(S_n)_{n\in \NZ}$ are disjoint we obtain 
				\[
					0\leq {^\mathcal{P} \left(\alpha \mathbb{1}_{\stsetG{S_A}}\right)}
					= \sum_{n=1}^\infty {^\mathcal{P} \left(\alpha \mathbb{1}_{\stsetG{S_n}}\right)}
					= \sum_{n=1}^\infty {^\mathcal{P} \left(\alpha \right)}\mathbb{1}_{\stsetG{S_n}},
				\]
				where we have used in the last step that $\mathbb{1}_{\stsetG{S_n}}$ is predictable. Hence, we obtain at the totally 
				$\mathcal{P}$-inaccessible $\FA$-stopping time $T$ that a.s.
				\begin{align}\label{Main:779}
					0\leq {^\mathcal{P} \left(\alpha \mathbb{1}_{\stsetG{S_A}}\right)}_T
					= \sum_{n=1}^\infty {^\mathcal{P} \left(\alpha \right)}\mathbb{1}_{\stsetG{S_n}}(T)=0.
				\end{align}
				Combining \eqref{Main:780} with \eqref{Main:781} and \eqref{Main:779} gives us
				\[
					\left({^\mathcal{P}C}\right)_T=\alpha \mathbb{1}_{\stsetO{S}{\infty}}(T)=C_{T-}=C_T,
				\]
				where the final identity follows from Theorem \ref{meyer_acc_Thm_2} as $T$ is  totally $\Lambda$-inaccessible. 
			\end{proof}
	
	\subsubsection{A Riesz representation} Let us conclude this section with a Riesz representation stated in \cite{EL80}, p.516,
	which was originally proven in the optional case in
	\citing{DM82}{Theorem 2}{184}.
	
	We denote by $\mathbb{G}$ a $\wedge$-stable
	vector space of processes, which satisfies the following:
	\begin{enumerate}[(i)]
		\item $\mathbb{G}$ contains the almost constant 
					processes, i.e. processes
					of the form $a\mathbb{1}_{(0,\infty]}$, $a\in \mathbb{R}$.
		\item All $Z\in \mathbb{G}$ are c\`agl\`ad with a limit
					at infinity such that $Z_+$ is $\Lambda$-measurable.
		\item	For any $\Lambda$-stopping time $T$ the process
		$\mathbb{1}_{\stsetO{T}{\infty}}$ is contained in $\mathbb{G}$.
	\end{enumerate}

	Then we have the following result:
	
	\begin{Thm}[\citing{EL80}{Theorem 18}{516}]
		Let $J$ be a positive linear form on $\mathbb{G}$ with the following
		property: For any non-increasing sequence $(Z^n)_{n\in \mathbb{N}}$
		of positive elements of $\mathbb{G}$, such that 
		\[
			\lim_{n\rightarrow\infty} \sup_{t\in [0,\infty]} Z^n_t=0
		\]
		we have  
		\[
			\lim_{n\rightarrow\infty} J(Z^n)=0.
		\]
		Then there exists two increasing, right-continuous processes
		$A$, $B$ with $\mathbb{E}[A_\infty]<\infty$,
		and $\mathbb{E}[B_\infty]<\infty$,
		where $A$ is $\mathcal{F}^\Lambda$-predictable, $A_0=0$
		and $B$ is $\Lambda$-measurable, purely discontinuous 
		with $\lim_{t\rightarrow \infty} B_{t}=B_{\infty}$ such that
		for any $Z\in \mathbb{G}$ we have
		\[
			J(Z)=\mathbb{E}\left[\int_{(0,\infty]} Z_s\mathrm{d} A_s
						+\int_{[0,\infty)} Z_{s+}\mathrm{d} B_s\right].
		\]
		The representating processes $A,B$ are unique up to indistinguishability.
	\end{Thm}


\subsection{El Karoui's general theory of optimal stopping}\label{sec:optstop}


	Let us henceforth consider a fixed $\mathbb{P}$-complete 
	Meyer-$\sigma$-field $\Lambda\subset \mathbb{F}\otimes
	\mathcal{B}([0,\infty))$
	with a given complete probability space
	$(\Omega,\mathbb{F},\mathbb{P})$. The upcoming part is based on \cite{EK81}, p.117-143, and we will just give
	a summary of the results stated in there for the special case where the considered ``weak chronology'' is given by
	$\stm$.

	\subsubsection{$\Lambda$-supermartingales}
	\label{optstop_ssec_1}	
		
		In this section we develop the notion of supermartingales
		for Meyer-$\sigma$-fields. As we will
		only state a small part of the results on such $\Lambda$-supermartingales
		we refer the interested reader to \cite{EL80},
		Chapter~III, for a full account.
		
		\begin{Def}[\citing{EK81}{Definition 2.25.2}{121}]
				\label{optstop_def_1}
				A family of random variables $(Z(S))_{S\in 
				\stm}$
				is a $\stm$\emph{-system}, if 
				\begin{enumerate}[(i)]
					\item For all $S$ and $T$ in 
					$\stm$,
					we have $Z(S)=Z(T)$ a.s.
					on $\{S=T\}$.
					\item $Z(S)$ is $\mathcal{F}^\Lambda_S$-measurable
					for all $S\in \stm$. 
				\end{enumerate}
				A process $Z:\Omega\times [0,\infty]\rightarrow
				\mathbb{R}$ aggregates a given
				$\stm$-system 
				$(Z(S))_{S\in \stm}$, if it
				is $\Lambda$-measurable and $Z_S=Z(S)$ almost surely 
				for all $S\in \stm$.
		\end{Def}
		
		\begin{Rem}[\cite{EK81}, p.121]\label{optstop_rem_1}
			The process which aggregates 
			$(Z(S))_{S\in \stm}$
			is unique up to indistinguishability by Corollary 
			\ref{meyer_cor_1}.
		\end{Rem}		
		
		Next we introduce the notion of super- and
		submartingales for the previous sets of random variables.		
		
		\begin{Def}[\cite{EK81}{Definition 2.25.3}{122}]
				\label{optstop_def_2}
				\begin{enumerate}[a)]
					\item We call an $\stm$-system 
								$(Z(S))_{S\in \stm}$
								an $\stm$-supermartingale system, if
								\begin{enumerate}[(i)]
									\item $Z(S)$ is integrable for 
									all $S\in \stm$.
									\item $Z(S)\geq \mathbb{E}[Z(T)|
									\mathcal{F}^\Lambda_S]$ a.s. for all $S,T\in
									\stm$ 
									with $S\leq T$.
								\end{enumerate}
								Analogously we define 
								$\stm$-martingale systems, 
								if we have equality in (ii). 
					\item A $\Lambda$-measurable process $Z:\Omega\times [0,\infty]
								\rightarrow \mathbb{R}$ is called a 
								$\Lambda$-supermartingale
								if the $\stm$-system $(Z_S)_{S\in
								\stm}$ is an 
								$\stm$-supermartingale system.
								Analogously, 
								we call $Z$ a $\Lambda$-martingale if $(Z_S)_{S
								\in \stm}$ is an 
								$\stm$-martingale system.
				\end{enumerate}
		\end{Def}		
				
		Next we get a statement concerning aggregation and decomposition 
		of $\Lambda$-super\-martingales. 	
		
		\begin{Pro}[Compare \citing{EK81}{Proposition 2.26}{123}]
		\label{optstop_pro_1}
			Let $(Z(S))_{S\in \stm}$ 
			be an $\stm$-supermartingale system.
			\begin{enumerate}[(i)]
				\item There exists a $\Lambda$-supermartingale $Z$
							unique up to indistinguishability,
							which aggregates $(Z(S))_{S\in \stm}$,
							i.e. 
							for all $S\in \stm$ we have 
							$Z_S=Z(S)$ almost surely.						
				\item Assume that the $\stm$-system
				$(Z(S))_{S\in \stm}$ 
				is of class($\text{D}^\Lambda$), i.e. $\{Z(S)|S\in \stm\}$ is uniformly integrable.				
						Then the 
							$\Lambda$-supermartingale $Z$ from (i) is of
							class($\text{D}^\Lambda$), i.e. 
								$(Z_S)_{S\in \stm}$ 
				is of class($\text{D}^\Lambda$), and it has the following
							unique decomposition
							\[
								Z=M-A-B_-,
							\]
							where $M:\Omega\times [0,\infty]\rightarrow \mathbb{R}$ 
							is a $\Lambda$-martingale of class($D^\Lambda$), 
							$A$ is a non-decreasing, right-continuous process
							which is $\mathcal{F}^\Lambda$-predictable with $A_0=0$, 
							$\mathbb{E}[A_\infty]
							<\infty$ and
							$B$ is a non-decreasing,
							 right-continuous, $\Lambda$-measurable process
							which is purely
							discontinuous, $B_{0-}=0$, $B_\infty=B_{\infty-}$, and $\mathbb{E}[B_\infty]
							<\infty$.							
			\end{enumerate}
		\end{Pro}

		\emph{Notation:} If a process $Z$ has a left- or a right limit,
		then we define by $Z_+$ the right-limit process and
		by $Z_-$ the left-limit process $Z_{\infty+}:=Z_\infty$.
		More generally, if a process $Z$ is just $\Lambda$-measurable,
		we will often use for $t\in [0,\infty)$ the notation
		\begin{align}\label{optstop_eq_5}
				\lsr{Z}{t}(\omega)&:=\limsup_{s\downarrow t} Z_s(\omega)
						:=\lim_{n\rightarrow \infty} \sup_{s\in (t,t+\frac1n)} 
						Z_s(\omega),  &\lsr{Z}{\infty}(\omega)
						:=Z_{\infty}(\omega),\nonumber\\
						Z_{t\ast}(\omega)&:=\liminf_{s\downarrow t} Z_s(\omega)
						:=\lim_{n\rightarrow \infty} \inf_{s\in (t,t+\frac1n)} 
						Z_s(\omega),  &{Z}_{\infty\ast}(\omega)
						:=Z_{\infty}(\omega),
			\end{align}
			and, for $t\in (0,\infty)$,
			\begin{align}\label{optstop_eq_6}
				&\lsl{Z}_t(\omega):=\limsup_{s\uparrow t} Z_s(\omega)
						:=\lim_{n\rightarrow \infty} \sup_{s\in (t-\frac1n,t)
						\cap [0,\infty)} Z_s(\omega),\nonumber \\
					& \lsl{Z}_{0}(\omega):=Z_0(\omega), \quad
					\lsl{Z}_{\infty}(\omega):=\limsup_{t
					\uparrow \infty} Z_t(\omega):=\lim_{n\rightarrow \infty}
					\sup_{s\in [n,\infty)} Z_s(\omega),\nonumber\\					
					&{_\ast Z}_t(\omega):=\liminf_{s\uparrow t} Z_s(\omega)
						:=\lim_{n\rightarrow \infty} \inf_{s\in (t-\frac1n,t)
						\cap [0,\infty)} Z_s(\omega),\nonumber \\
					& {_\ast Z}_{0}(\omega):=Z_0(\omega), 
					\quad {_\ast Z}_{\infty}(\omega):=\liminf_{t
					\uparrow \infty} Z_t(\omega):=\lim_{n\rightarrow \infty}
					\inf_{s\in [n,\infty)} Z_s(\omega)
			\end{align}				
		Furthermore we denote by $^ \Lambda Z$ the $\Lambda$-projection 
		and by $^\mathcal{P} Z$ the $\mathcal{F}^\Lambda$-predictable projection
		of $Z$ if they exist; we furthermore follow the convention $^ \Lambda Z_{\infty}:=Z_{\infty}$.
		
		\begin{Rem}\label{optstop_rem_3}
				By \citing{DM78}{Theorem 90}{143}, the process $\lsr{Z}{}$ is $\mathcal{F}^\Lambda$-progressively
				measurable and $\lsl{Z}$ is an 
				$\mathcal{F}^\Lambda$-predictable process.
				In general, $\lsr{Z}{}$ is not 
				$\Lambda$-measurable, not even $\mathcal{F}^\Lambda$-optional,
				anymore, which can be seen by \citing{DM78}{Remark 91 (b)}{144}.
		\end{Rem}
		
		In the last proposition of this section we state results
		on the path of a $\Lambda$-supermartingale 
		and detailed results on its points of discontinuity.
		
		\begin{Pro}[Compare \citing{EK81}{Proposition 2.27}{125}]\label{optstop_pro_2}
			\begin{enumerate}[(i)]
				\item We have that every $\Lambda$-martingale 
				$M:[0,\infty]\rightarrow \mathbb{R}$ is l\`adl\`ag. 
				The process 
			$M_+$ is an $\mathcal{F}^\Lambda_+$-martingale whose
			$\Lambda$-projection is equal to $M$, i.e.
			\[
				M={^\Lambda\left(M_+\right)}.
			\]
				\item		Every $\Lambda$-supermartingale 
				$Z$ of class($\text{D}^\Lambda$) is l\`adl\`ag. 
				Furthermore such a $Z$ is the
			$\Lambda$-projection of an 
			$\mathcal{O}(\mathcal{F}^\Lambda_+)$-supermartingale
			$\hat{Z}$ and
			\[
				\hat{Z}_+=Z_+.
			\]
			For the discontinuities $\Delta A:=A-A_-$
			and $\Delta B:=B-B_-$ of the non-decreasing, right-continuous processes $A$ and $B$ of 
			the decomposition $Z=M-A-B_-$ from Proposition 
			\ref{optstop_pro_1}, we have
			\[
				\Delta A=Z_--{^\mathcal{P} Z},\quad
				\Delta B=Z- {^\Lambda (Z_+)}.
			\]
			In particular, $Z_-\geq {^\mathcal{P} Z}$ and $Z\geq {^\Lambda(Z_+)}$.
			\end{enumerate}
		\end{Pro}		
		
	\subsubsection{Snell envelope}\label{optstop_ssec_2}		
		
		Next we introduce a classical process in
		the context of optimal stopping:
		
		\begin{Thm}[\citing{EK81}{Theorem 2.28}{126}]
			\label{optstop_thm_1}
			Let $(Z(T))_{T\in \stm}$ be a positive
				$\stm$-system.
				The maximal conditional gain
				\[
					\bar{Z}(S):=\esssup_{T\geq S,\, T\in \stm}
					\mathbb{E}\left[Z(T)| \mathcal{F}^\Lambda_S\right],\quad
					S\in \stm,
				\]
				is an $\stm$-supermartingale system,
				which is
				aggregated by a $\Lambda$-supermartingale $\bar{Z}$. This
				$\bar{Z}$ is the smallest among all positive 
				$\Lambda$-supermartingales $\tilde{Z}$ dominating $Z$ in the 
				sense that $\tilde{Z}_T\geq Z(T)$ for all $T\in 
				\stm$.
		\end{Thm}
		
		\begin{Def}[\citing{EK81}{Remark}{127}]\label{optstop_def_3}
				The process $\bar{Z}$ is called the 
				\emph{$\Lambda$-Snell envelope or just 
				Snell envelope of the $\stm$-system 
				$(Z(T))_{T\in \stm}$}.
		\end{Def}
		
		As it is important to know in which situations $\bar{Z}$ is of
		class($D^\Lambda$) we need the following result.
		
		\begin{Pro}[\citing{EK81}{Proposition 2.29}{127}]
			\label{optstop_pro_3}
			If the given $\stm$-system
			$(Z(T))_{T\in \stm}$
			of Theorem \ref{optstop_thm_1} 
			is of class($\text{D}^\Lambda$),
			then also its Snell envelope $\bar{Z}$ is of
			class($\text{D}^{\Lambda}$).
			In that case, $\bar{Z}$ has the decomposition
			$\bar{Z}=\bar{M}-\bar{A}-\bar{B}_-$, with processes
			$\bar{M},\bar{A},\bar{B}$ defined
			as in Proposition
			\ref{optstop_pro_1} (ii).
		\end{Pro}			
		
		
	\subsubsection{Optimality criterion and an approximation
			of the Snell envelope}\label{optstop_ssec_3}			
		
		For an arbitrary positive 
		$\stm$-system 
			$(Z(S))_{S\in \stm}$ one can formulate
		the following optimal stopping problem
		\begin{align}\label{optstop_eq_1}
			\text{Maximize }\quad \mathbb{E}[Z(S)] \quad \text{ over all }
			S\in \stm.
		\end{align}
		The following theorem uses the Snell envelope to give necessary and sufficient
		conditions for  a stopping time to be optimal, i.e. to attain the maximal value in \eqref{optstop_eq_1}.
	
		\begin{Thm}[\citing{EK81}{Theorem 2.31}{129}]
			\label{optstop_thm_2}			
				Let $(Z(S))_{S\in \stm}$ be a positive
				$\stm$-system of
				class($\text{D}^{\Lambda}$)
				and let $\bar{Z}$ denote its $\Lambda$-Snell envelope.
				Then $\bar{U}\in \stm$ is 
				optimal for \eqref{optstop_eq_1} if and only if
				\begin{enumerate}[(i)]
					\item $Z(\bar{U})=\bar{Z}_{\bar{U}}$ $\mathbb{P}$-a.s.,
					\item $(\bar{Z}_{t\wedge \bar{U}})_{t\in [0,\infty]}$
					is a $\Lambda$-martingale.
				\end{enumerate}
		\end{Thm}
		
		The next proposition introduces  the
		entry time of the event that the
		Snell envelope is close to the
		$\stm$-system
		and gives more precise
		information about the processes $\bar{A}$ and $\bar{B}$
		of the decomposition of the Snell envelope
		of a $\Lambda$-measurable process introduced in Proposition
		\ref{optstop_pro_3}.						
						
		\begin{Pro}[Compare \citing{EK81}{Proposition 2.32}{130} and \citing{EK81}{Proposition 2.34}{131}]
				\label{optstop_pro_4}
				Let $(Z(T))_{T\in \stm}$ be a positive
				$\stm$-system of
				class($\text{D}^{\Lambda}$), aggregated by 
				a $\Lambda$-measurable process $Z$ 
				and denote by $\bar{Z}$ its $\Lambda$-Snell envelope with
				$\bar{M}-\bar{A}-\bar{B}_-$ the decomposition of $\bar{Z}$
				from Proposition \ref{optstop_pro_3}.
				Furthermore consider for $\lambda\in (0,1)$ the
				$\Lambda$-measurable set				
					\[
						E^\lambda:=\left\{(\omega,t)\in\Omega\times 
						[0,\infty)\ \middle\vert\  \lambda 
						\bar{Z}_t(\omega)\leq Z_t(\omega)\right\}.
					\]
				and let
				\[
					T_{S}^\lambda(\omega):=\inf\left\{t\geq S(\omega)\,
					\middle| \, (\omega,t)\in E^\lambda\right\}
				\]					
					denote the entry time of $E^\lambda$ after some given $S \in \stm$.					
					Then we have 
					\[
						\mathbb{E}\left[\bar{Z}_S\right]=\mathbb{E}
						\left[\bar{Z}_{T_S^\lambda} \mathbb{1}_{\{T_S^\lambda
						\in E^\lambda\}}
					+ \bar{Z}_{T_S^\lambda+} \mathbb{1}_{\{T_S^\lambda 
					\notin E^\lambda\}}\right]
					\]
					and
				\begin{align*}
					\bar{A}_S=\bar{A}_{T_S^\lambda},\quad
					\bar{B}_{S-}=\bar{B}_{T_S^\lambda-}
						\mathbb{1}_{\{T_S^\lambda
						\in E^\lambda\}}
					+\bar{B}_{T_S^\lambda}\mathbb{1}_{\{T_S^\lambda
						\notin E^\lambda\}}\quad \text{ a.s.}
				\end{align*}
				In particular, we have up to evanescent sets that
				$\{\bar{A}>	\bar{A}_-\}
				\subset \{\bar{Z}_-=\lsl{Z}\}$, 
				$\{\bar{B}>\bar{B}_-\}\subset \{\bar{Z}=Z\}$ and
				\begin{align*}
					\bar{Z}={^\Lambda (\bar{Z}_+)\vee Z} \quad \text{ and }
					\quad \bar{Z}_-=(^\mathcal{P} \bar{Z})
					\vee \lsl{Z}.
				\end{align*}
		\end{Pro}
		
		\begin{Rem}
			The stopping time $T_S^\lambda$ from the previous 
			proposition may not be
			a $\Lambda$-stopping time, very much like a level passage time for a predictable process may not itself be predictable.
		\end{Rem}
			

	\subsubsection{Relaxed optimal stopping}\label{optstop_ssec_4}
		
		In this subsection we state results on some stopping time
		which ``nearly''
		 solves the optimal stopping 
		problem introduced in \eqref{optstop_eq_1}. This result
		will give rise to the notion of divided stopping times
		to be discussed thereafter.
		
		For now we assume that the $\stm$-system 
		$(Z(T))_{T\in \stm}$ can be aggregated by some
		 $\Lambda$-measurable
		process $Z$. 		
		
		\begin{Pro}[\citing{EK81}{Proposition 2.35 and 2.36}{133 and 135} and \citing{EK81}{Remark}{136}]\label{optstop_pro_5}
			We use the notations and hypotheses from Proposition
			\ref{optstop_pro_3} and \ref{optstop_pro_4}. For any
			$S\in \stm$, 
			the family
			 of $\mathcal{F}^\Lambda_+$-stopping times 
			$(T_S^\lambda)_{\lambda\in [0,1)}$ is non-decreasing in $\lambda\in [0,1)$
			and we denote its limit by
			  $T_S:=\lim_{\lambda
			 \uparrow 1} T_{S}^\lambda$. We have 
			\begin{align*}
				H_S^-&:=\left\{T_S^\lambda< T_S
				\text{ for every $\lambda\in [0,1)$}\right\}
				\subset \left\{ \bar{Z}_{T_S-} = \lsl{Z}_{T_S}\right\},\\
				H_S&:=(H_S^-)^c\cap\left\{\bar{Z}_{T_S} \leq Z_{T_S}
				\right\}\subset \left\{ \bar{Z}_{T_S}= Z_{T_S}\right\},\\
				H_S^+&:=(H_S^-)^c\cap\left\{\bar{Z}_{T_S} > Z_{T_S} \right\}
				\subset \left\{ \bar{Z}_{T_S+}= \lsr{Z}{T_S}\right\}
			\end{align*}
			and
			\begin{align}\label{Main:140}
				T_S=\inf\left\{t\geq S \ \middle\vert\  Z_t=\bar{Z}_t
											\text{ or } \lsl{Z}_t=\bar{Z}_{t-}
											\text{ or } \lsr{Z}{t}=\bar{Z}_{t+}  \right\}.
			\end{align}
			The sets $H_S$ and $H_S^+$ are contained in $\mathcal{F}_{T_S}^\Lambda$ and
			$H_S^-\in \mathcal{F}_{T_S-}^\Lambda$. Moreover $(T_S)_{H_S^-}$ is a predictable 
			$\mathcal{F}^\Lambda_+$-stopping time, $(T_S)_{H_S}$ 
			is a $\Lambda$-stopping time, $(T_S)_{H_S^+}$ 
			is an $\mathcal{F}^\Lambda_+$-stopping time and for 
			each $S\in \stm$ we get that
			\begin{align}\label{Main:230}
				\bar{Z}_S=\mathbb{E}\left[\lsl{Z}_{T_S}\mathbb{1}_{H_S^-}+Z_{T_S}
				\mathbb{1}_{H_S}
								+\lsr{Z}{T_S}\mathbb{1}_{H_S^+}\ 
								\middle\vert\  \mathcal{F}^\Lambda_S\right]
			\end{align}
			 and, in particular,
			\begin{align}\label{optstop_eq_2}
				\mathbb{E}[\bar{Z}_S]
				= \mathbb{E}\left[\lsl{Z}_{T_S}\mathbb{1}_{H_S^-}+Z_{T_S}
				\mathbb{1}_{H_S}
								+\lsr{Z}{T_S}\mathbb{1}_{H_S^+}\right].
			\end{align}
		\end{Pro}
		

\subsubsection{General Divided Stopping Times}\label{optstop_ssec_5}
	
	Even in deterministic examples it is easy to see that
	it is not always possible to solve the optimal
	stopping problem \eqref{optstop_eq_1}. 
	Proposition \ref{optstop_pro_5} though gives a good 
	idea how to relax this problem suitably. 
	
	\begin{Def}[\citing{EK81}{Definition~2.37}{136-137}]\label{optstop_def_4}
			A quadruple $\sigma:=(T,W^-,W,W^+)$ is called a \emph{divided
			stopping time}, if $T$ is an $\mathcal{F}^\Lambda_+$-stopping time
			and $W^-,W,W^+$ form a partition 
			of $\Omega$ such that
			\begin{enumerate}[(i)]
				\item $W^-\in 
				(\mathcal{F}^\Lambda_+)_{T-}$ and $W^-\cap \{T=0\}=\emptyset$,
				\item $W\in \mathcal{F}^\Lambda_T$,
				\item $W^+ \in \mathcal{F}_{T+}^\Lambda$ and $W^+\cap 
				\{T=\infty\}=\emptyset$,
				\item $T_{W^-}$ is an $\mathcal{F}^\Lambda_+$-predictable 
				stopping time,
				\item $T_{W}$ is a $\Lambda$-stopping time.
			\end{enumerate}
			The set of all divided stopping times will be denoted 
			as $\stmd$.
			For a $\Lambda$-measurable positive process $Z$,
			we define the value attained 
			at a divided stopping time $\sigma=(T,W^-,W,W^+)$ as
			$$Z_\sigma:=\lsl{Z}_{T}\mathbb{1}_{W^-}+Z_T\mathbb{1}_{W}
				+\lsr{Z}{T}\mathbb{1}_{W^+}.$$
	\end{Def}
	
	\begin{Rem}\label{optstop_rem_4}
			Proposition \ref{optstop_pro_5} shows that
			$\delta_S:=(T_S,H_S^-,H_S,H_S^+)$ is a divided stopping time.
	\end{Rem}
	
	The $\Lambda$-(super)martingale property can be extended to accommodate divided stopping times:
	
	\begin{Lem}[\citing{EK81}{Lemma 2.38}{137}]\label{optstop_pro_6}
			Let $\sigma=(T,W^-,W,W^+)$ be a divided stopping time
			and $S$ a $\Lambda$-stopping time
			such that $\sigma\geq S$, i.e. $T\geq S$ and $T>S$ on $W^-$.
			Then we have for every positive 
			$\Lambda$-martingale $M:\Omega\times [0,\infty]
			\rightarrow \mathbb{R}$
			\[
				M_S=\mathbb{E}\left[M_\sigma\ \middle\vert\ 
				\mathcal{F}^\Lambda_S\right]
			\]
			and for every positive $\Lambda$-supermartingale 
			$Z:\Omega\times [0,\infty]\rightarrow \mathbb{R}$ that
			\[
				Z_S\geq \mathbb{E}\left[Z_\sigma\
				\middle\vert\  \mathcal{F}^\Lambda_S\right].
			\]
	\end{Lem}
	
	By contrast to problem \eqref{optstop_eq_1} its relaxation for 
	divided stopping times 
	\begin{align}\label{optstop_eq_3}
			\text{Maximize }\quad \mathbb{E}[Z_\sigma] \quad 
			\text{ over } \quad \sigma\in \stmd,
	\end{align}
	always hase a solution:
	
	\begin{Thm}[\citing{EK81}{Theorem 2.39}{138}]\label{optstop_thm_3}
		Let $Z$ be a positive $\Lambda$-measurable process
		of class(D$^\Lambda$) and $\bar{Z}$ its Snell envelope.
		Then for every $S\in \stm$
		\begin{align}\label{optstop_eq_7}
			\mathbb{E}\left[\bar{Z}_S\right]=\mathbb{E}
			\left[\bar{Z}_{\delta_S}\right]=\mathbb{E}\left[Z_{\delta_S}\right]
			=\sup_{\sigma\geq S,\, \sigma \in \stmd}
			\mathbb{E}\left[Z_\sigma\right],
		\end{align}
		where $\delta_S:=(T_S,H_S^-,H_S,H_S^+)$ is the divided 
		stopping time given by
		Proposition \ref{optstop_pro_5} and where the supremum
		is taken over all divided
		stopping times $\sigma=(T,W^-,W,W^+)$ such that $T\geq S$
		and $T>S$ on $W^-$. In particular,
		the divided stopping time $\delta_0$ is optimal 
		for \eqref{optstop_eq_3}.
	\end{Thm}

	The optimal divided stopping time $\delta_0$ is 
	constructed with the help of $T_0$, which is constructed by
	using part (i) of Theorem \ref{optstop_thm_2}.
	It is also possible to construct a second optimal
	divided stopping time with the help of the second condition
	as the next theorem shows.
		
	\begin{Thm}[\citing{EK81}{Theorem 2.40}{138}]\label{optstop_thm_4}
			Let $Z$ be a positive $\Lambda$-measurable process 
			of class(D$^\Lambda$) and denote by 
			$\bar{Z}=\bar{M}-\bar{A}-\bar{B}_-$ its Snell envelope
			with decomposition from Proposition
			\ref{optstop_pro_3}.
			Define 
			\begin{align}\label{Main:141}
				T^S:=\inf \left\{u\geq S\ \middle\vert\  \bar{A}_u+\bar{B}_u
				>\bar{A}_{S}+\bar{B}_{S-}\right\}
			\end{align}
			 and the $\mathcal{F}^\Lambda_{T^S+}$-measurable sets
			\begin{align*}
				K_S^-&:=\{\bar{A}_{T^S}>\bar{A}_S\},\\
				K_S&:=\left\{\bar{A}_{T^S}=\bar{A}_S,
				\quad \bar{B}_{T^S}>\bar{B}_{S-}\right\},\\
				K_S^+&:=\left\{\bar{A}_{T^S}+\bar{B}_{T^S}
				=\bar{A}_{S}+\bar{B}_{S-}\right\}.
			\end{align*}
			The quadruple $\sigma_S:=(T^S,K_S^-,K_S,K_S^+)$ is a 
			divided stopping time satisfying
			\begin{enumerate}[(i)]
				\item $\bar{Z}_{\sigma_S}=Z_{\sigma_S}$,
				\item $\mathbb{E}[\bar{Z}_S]=\mathbb{E}[Z_{\sigma_S}]$ 
				for all $S\in \stm$.
			\end{enumerate}
			In particular, the divided stopping time $\sigma_0$ is 
			optimal for \eqref{optstop_eq_3}.
	\end{Thm}
	
	\begin{Rem}[\citing{EK81}{Remark}{140}]\label{optstop_rem_5}
			The optimality conditions given in Theorem 
			\ref{optstop_thm_2} show that no optimal  stopping time for
			\eqref{optstop_eq_1} can be smaller than $\delta_0$ and none can be larger
			than $\sigma_0$. 
			Indeed, $\bar{Z}$ looses the martingale property after $T^0$ and
			therefore $T^0$ is by Theorem \ref{optstop_thm_2} 
			dominating all possible optimal stopping times.
			Taking a closer look at the set $K_0^-$, we see that actually on this set $\bar{Z}$ has
			already lost the martingale property at $T^0$ and, 
			therefore, $T^0$ is strictly larger than all optimal
			stopping times for \eqref{optstop_eq_1} on this set.
			On the other hand, we see 
			that $T_0$ is smaller than the 
			entry time of the set $\{\bar{Z}=Z\}$, which
			we denote by $\bar{T}_0$,
			and $T_0$ is strictly smaller than $\bar{T}_0$ on $
			H_0^-\cap \{T_0<\bar{T}_0\}$.
			Hence, the divided stopping time 
			$(T_0,H_0^-\cap \{T_0<\bar{T}_0\},(H_0^-\cap
			\{T_0=\bar{T}_0\})\cup H_0,H_0^+)$ is smaller 
			than all optimal stopping times for \eqref{optstop_eq_1}.
	\end{Rem}


\subsubsection{Conditions for optimality
 in the optional case}\label{optstop_ssec_6}
    In this section, let us consider the classical case where $\Lambda=\mathcal{O}(\FA)$ is the optional $\sigma$-field of a right-continuous filtration $(\FA_t)_{t\geq 0}$ with $\FA_\infty:=\bigvee_{t\geq 0}\FA_t$ and $\FA_{0-}\subset \FA_0$ a $\WM$-complete $\sigma$-field.
	Then $\stm$ coincides with the set of
	``classical'' stopping times with respect to 
	$(\mathcal{F}_t)_{t\geq 0}$, which we denote by
	$\st$. 
	Furthermore we assume, that $Z$ is an optional, positive 
	process of class(D). The optimal stopping problem is then	to
	\begin{align}\label{optstop_eq_4}
		\text{Maximize} \quad \mathbb{E}[Z_T] \quad \text{ over } 
		\quad T\in \st
	\end{align}
	and we get the following first optimality
	result:
	
	\begin{Thm}[\citing{EK81}{Theorem 2.41}{140}]
			\label{optstop_thm_5}
			Let $\Lambda=\mathcal{O}(\mathcal{F})$. Assume $Z$ is an optional,
			positive process of class(D) and denote by $\bar{Z}$ its Snell envelope. 
			\begin{enumerate}[(i)]
				\item If the process $Z$ is upper-semicontinuous 
						from the right and from the left, i.e.
						$\lsl{Z}\leq Z$ and $\lsr{Z}{}\leq Z$, 
						then the entry time $\bar{T}_0$ of the set
						$\{Z=\bar{Z}\}$ is optimal for \eqref{optstop_eq_4} and 
						it is equal to $T_0$ from \eqref{Main:140}.
						In  particular,					
						it is the smallest optimal stopping time.
				\item If $\bar{Z}$ satisfies $\bar{Z}_-={^\mathcal{P} \bar{Z}}$
				and $Z$ is upper-semicontinuous from the right, then
				the entry time $\bar{T}^0$ of the set $\{\bar{M}\neq \bar{Z}\}$
				is optimal for \eqref{optstop_eq_4} 
				with $\bar{M}$ from Proposition \ref{optstop_pro_3} and
				$\bar{T}^0=T^0$ with $T^0$ from
				\eqref{Main:141}. In particular,
				 $\bar{T}^0$ is the largest optimal stopping time.
			\end{enumerate}
	\end{Thm}
	
	\begin{Rem}\label{optstop_rem_6}
	\begin{enumerate}[(i)]
	    \item 
			The proof of the previous result is mainly based on
			the results on divided stopping times, which can also be
			used for general Meyer-$\sigma$-fields.
			But the reason why we cannot get the same general results
			for those more general $\sigma$-fields
			is that $T_S$ and $T^S$ are not necessarily Meyer-stopping times and, therefore we cannot use the $\Lambda$-supermartingale 
			property of $\bar{Z}$ in the proof. 
			\item The condition $\bar{Z}_-={^\mathcal{P}\bar{Z}}$, in part (ii) of the previous theorem, means that the Snell envelope $\bar{Z}$ of $Z$ has to be left-upper-semicontinuous in expectation (see Lemma \ref{extension_lem_2} (ii) below). It will hold when $Z$ is upper-semicontinuous in expectation. Note that $\bar{Z}_-\leq {^\mathcal{P}\bar{Z}}$ is always fulfilled since $\bar{Z}$ is a supermartingale. 
			\end{enumerate}
	\end{Rem}
	
	The following result gives a nice equivalence for
	the previous pathwise properties of $Z$:
	 
	\begin{Pro}[Compare \citing{BS77}{Theorem II.1}{305}]
	\label{optstop_pro_7}
		An optional positive process $Z:\Omega\times
		[0,\infty)\rightarrow \mathbb{R}$ of class(D) satisfies
		\[
			\mathbb{E}[Z_T]\geq \limsup_{n\rightarrow \infty}
			\mathbb{E}[Z_{T_n}]
		\]
		for every monotone sequences $T_n$ of stopping times 
		converging to $T$ if and only if
		$Z$ has upper-semicontinuous paths from the right on
		$[0,\infty)$ and 
		\[
			^\mathcal{P} Z_t\geq \lsl{Z}_t \quad \text{for $t\in (0,\infty]$
			almost surely.}
		\]
	\end{Pro}	
	
	\begin{Rem}
		\cite{BS77} state in the introduction to chapter I, p.301, that we can use processes $Z$ of class$(D)$ in their Theorem II.1, p.305, instead of merely bounded processes
	\end{Rem}
	
	The previous proposition gives us the second optimality result:
	
	\begin{Thm}[\citing{EK81}{Theorem 2.43}{142}]
			\label{optstop_thm_6}
				Let $Z$ be an optional, positive process of class(D).
				\begin{enumerate}[(i)]
					\item 	Assume $Z$ is 
					\emph{upper-semicontinuous in expectation}, i.e. 
							for every monotone (not necessarily strict) 
							sequence $(T_n)_{n\in \mathbb{N}}$ with limit $T$, we have
							\[
									\mathbb{E}[Z_T]\geq \limsup_{n\rightarrow \infty} 
									\mathbb{E}[Z_{T_n}].
							\]					 
							 Then the entry time $\bar{T}_0$ of the set 
							$\{Z=\bar{Z}\}$ is optimal as is the entry time 
							$\bar{T}^0$ of the set 
							$\{\bar{M}\neq \bar{Z}\}$ with $\bar{M}$ from the 
							decomposition
							of $\bar{Z}$ given
							by Proposition~\ref{optstop_pro_3}.
					\item If for every non-increasing sequence of stopping 
					times $(T_n)_{n\in \mathbb{N}}$ with limit $T$ we have
							\[
									\mathbb{E}[Z_T]\geq \limsup_{n\rightarrow \infty}
									\mathbb{E}[Z_{T_n}]
							\]
							and if for every non-decreasing sequence of stopping 
							times $(T_n)_{n\in \mathbb{N}}$ with limit $T$ we have
							\[
									\lim_{n\rightarrow \infty} \sup_{T_n\leq R\, 
									\in \st}\mathbb{E}[Z_R]
									=\sup_{T\leq R\, \in \st}\mathbb{E}[Z_R],
							\]
							then the entry time $\bar{T}^0$ of the set 
							$\{\bar{M}\neq \bar{Z}\}$ with $\bar{M}$ from the decomposition
							of $\bar{Z}$ given
							by Proposition \ref{optstop_pro_3} is optimal.
				\end{enumerate}
	\end{Thm}
	
	\begin{Rem}\label{optstop_rem_8}
			In the original article of \cite{EK81} it is claimed that part (i)
			of the previous theorem follows directly by
			the equivalence in Proposition \ref{optstop_pro_7}
			and the statement in Theorem \ref{optstop_thm_5}. But 
			as one can see we do not get upper-semi-continuity from the 
			left of the process $Z$, but just of the process
			$^\mathcal{P} Z$. These two processes are not the same in general and
			hence the proof is incomplete. This gap will be closed by 
			Proposition \ref{extension_pro_2} in the next section.		
	\end{Rem}


\section{Supplementary results to Lenglart's and to El Karoui's theory}\label{sec:extension}
		
		In this section we will prove some additional results which are not given in
		\cite{EK81}, but which we find useful. Throughout this 
		section we will fix a
		filtered probability space $(\Omega,\mathbb{F}, 
		\mathcal{F}:=(\mathcal{F}_t)_{t\geq 0},\mathbb{P})$
		and a Meyer-$\sigma$-field $\Lambda\subset 
		\mathbb{F}\otimes \mathcal{B}([0,\infty))$
		with $\mathcal{F}_{\infty}:=\bigvee_{t} \mathcal{F}_t\subset
		\mathbb{F}$,
		$\mathcal{F}_{0-}$ some $\WM$-complete $\sigma$-field contained in $\mathcal{F}_0$		and $\mathcal{F}$ is satisfying the usual conditions of completeness and 
		right-continuity.


	\subsection{Special case of an embedded Meyer-$\sigma$-field}
	\label{extension_ssec_1}
		
		\begin{Lem}\label{extension_lem_1}
			Assume that the given Meyer-$\sigma$-field $\Lambda$ is embedded
			in the sense that:
			\begin{align}\label{extension_eq_1}
				\mathcal{P}(\mathcal{F})\subset \Lambda\subset
			\mathcal{O}(\mathcal{F}),
			\end{align}
			where $\mathcal{P}(\mathcal{F})$ and
			$\mathcal{O}(\mathcal{F})$ denote, respectively, 
			the predictable and the
			optional $\sigma$-field associated with $\mathcal{F}$.			
			
			Then we have
			$\mathcal{F}_{t-}=\mathcal{F}^\Lambda_{t-}$ for $t> 0$,
			$\mathcal{F}^\Lambda_{t+}=\mathcal{F}_t$ for $t\geq 0$,
			and $\Lambda$ is a $\mathbb{P}$-complete 
			Meyer-$\sigma$-field. In particular, this gives us 
			$\mathcal{F}_{t-}\subset \mathcal{F}^\Lambda_t
			\subset \mathcal{F}_t$, $t > 0$,
			and $\mathcal{P}(\mathcal{F})=\mathcal{P}(\mathcal{F}^\Lambda_+)$,
			$\mathcal{O}(\mathcal{F})=\mathcal{O}(\mathcal{F}^\Lambda_+)$.		
		\end{Lem}
		
		\begin{proof}
				First of all, we have by Theorem \ref{meyer_thm_4}
				that $\Lambda$ is $\mathbb{P}$-complete by
				the properties of $\mathcal{F}$. Next we get
				by Example \ref{meyer_ex_2} and Remark \ref{meyer_rem_1} that
				\begin{align}\label{extension_eq_12}
					\mathcal{F}_{t-}
					={\mathcal{F}_t^{\mathcal{P}(\mathcal{F})}}\subset\mathcal{F}^\Lambda_t\subset{\mathcal{F}_t^{\mathcal{O}(\mathcal{F})}}= 
					\mathcal{F}_t
				\end{align}
				for $t> 0$.			
				Furthermore, we get for $t>0$ by \eqref{extension_eq_12}
				that 
				\begin{align*}
					\mathcal{F}_{t-}&=\sigma\left(\bigcup_{s\in [0,\infty),\, s<t}
					\mathcal{F}_s\right)
					=\sigma\left(\bigcup_{s\in [0,\infty),\, s<t}
					\mathcal{F}_{(s+\frac{t-s}{2})-}\right)\\
					&\overset{\eqref{extension_eq_12}}{\subset}
					\sigma\left(\bigcup_{s\in [0,\infty),\, s<t}
					\aFA_{(s+\frac{t-s}{2})}\right)
				\subset \mathcal{F}^\Lambda_{t-}
				\end{align*}
				and on the other hand we have
				\[
					\mathcal{F}^\Lambda_{t-}
					=\sigma\left(\bigcup_{s\in [0,\infty),\, s<t}
					\mathcal{F}_s^\Lambda\right)
					\subset \sigma\left(\bigcup_{s\in [0,\infty),\,s<t}
					\mathcal{F}_s\right)
					=\mathcal{F}_{t-}.
				\]
				Combining the previous two results proves $\mathcal{F}^\Lambda_{t-}=\mathcal{F}_{t-}$
				for $t> 0$.
				Additionally we see that for $t\geq 0$
				\[
					\mathcal{F}^\Lambda_{t+}
					=\bigcap_{r\in [0,\infty),\,t<r}\mathcal{F}_r^\Lambda 
					\subset \bigcap_{r\in [0,\infty),\,t<r} \mathcal{F}_r=
					\mathcal{F}_{t+}=\mathcal{F}_t
				\]
				and
				\begin{align*}
					\mathcal{F}_{t}
					\subset \bigcap_{r\in [0,\infty),\,t<r} \sigma\left(
					\bigcup_{s\in [0,\infty),\, s<r}\mathcal{F}_s\right)
					=\bigcap_{r\in [0,\infty),\,t<r} \mathcal{F}_{r-}
					\subset\bigcap_{r\in [0,\infty),\,t<r} \mathcal{F}^\Lambda_r
					=\mathcal{F}^\Lambda_{t+},
				\end{align*}
				which implies $\mathcal{F}_{t+}^\Lambda=\mathcal{F}_t$.
				As we have proven that the filtration $(\mathcal{F}_t)_{t\geq 0}$
				and $(\mathcal{F}^\Lambda_{t+})_{t\geq 0}$ are the same,
				we see that $\mathcal{O}(\mathcal{F})$ and 
				$\mathcal{O}(\mathcal{F}^\Lambda_+)$ 
				(resp. $\mathcal{P}(\mathcal{F})$ and 
				$\mathcal{P}(\mathcal{F}^\Lambda_+)$)
				are generated by the same processes, which shows in particular that they are
				the same.
		\end{proof}		
		

	\subsection{Approximating the limes superior}\label{extension_ssec_2}

    It is sometimes convenient to know that the $\limsup$ from the right or from the left of a stochastic process at a stopping time can be realized by a suitable sequence of larger or smaller stopping times converging to it. Such a result was proven 
	by \citing{DL82}{Lemma 1}{300}, 
	for the case of optional processes. We extend their argument to cover $\Lambda$-measurable processes:

	\begin{Pro}\label{extension_pro_1}
		Assume that $\Lambda$ is embedded in the
		sense of \eqref{extension_eq_1}.
		Let $Z$ be a $\Lambda$-measurable
		process with $Z_{\infty}=0$ and denote 
		by $\lsr{Z}{}$ and, respectively, $\lsl{Z}$ 
		the right- and the left-upper-semicontinuous envelope
		of $Z$,
		which are defined in \eqref{optstop_eq_5} and \eqref{optstop_eq_6}.
			Now we have the following two results:
			\begin{enumerate}[(i)]
				\item For every given $\mathcal{F}$-stopping time $T$, 
				there exists
							a non-increasing sequence $(T_n)_{n\in \mathbb{N}}
							\subset \stm$
						with $T_n\geq T$, $\infty>T_n>T$ on $\{T<\infty\}$ and
						$\lim_{n					\rightarrow \infty}T_n=T$ such that
							$\lsr{Z}{T}=\lim_{n\rightarrow \infty} 
							Z_{T_n}$
								almost surely.
				\item For any predictable $\mathcal{F}$-stopping time $T>0$,
							there exists a sequence of 
							$\Lambda$-stopping times 
							$(T_n)_{n\in \mathbb{N}}$ with
						$T_n< T$ and
						$\lim_{n \rightarrow \infty}T_n=T$ such that
							$\lsl{Z}_{T}=\lim_{n\rightarrow \infty} Z_{T_n}$ 
							almost surely.
			\end{enumerate}
		\end{Pro}		
			
			\begin{Rem}\label{extension_rem_1}
						\begin{enumerate}[(i)]
							\item Result (ii) of Proposition \ref{extension_pro_1}
									cannot be generalized to an 
									$\mathcal{F}$-stopping time $T$
									in the way it is possible for result~(i), because the sequence 
									$(T_n)_{n\in \mathbb{N}}$ is an announcing
									sequence, which would directly imply that 
									$T$ is predictable.
							\item Obviously one can get analogously to Proposition
							\ref{extension_pro_1} the same result for the right- and 
							left-lower-semicontinuous envelope of a given 
							$\Lambda$-measurable
							process $Z$ by using Proposition \ref{extension_pro_1}
							for $-Z$.
						\end{enumerate}
			\end{Rem}

	\begin{proof}
	\emph{Proof of (i):} 
		We will prove this result by adapting the proof of \citing{DL82}{Lemma 1}{300}. Assume without loss of generality that 
		$\mathbb{P}(T<\infty)>0$, because for $T=\infty$ a.s. we could 
		set $T_n=\infty$ as $Z_\infty=0=Z^\ast_\infty$.
		Next we see that $\lsr{Z}{}$ is $\mathcal{F}$-progressively
		measurable by \citing{DM82}{Theorem 90}{143}, and, 
		therefore, $\lsr{Z}{T}\in \mathcal{F}_{T}$ by \citing{DM82}{Theorem 64 (b)}{122}.
		 Furthermore, we can assume that
		$Z$  and $\lsr{Z}{T}$ are bounded by replacing $Z$ with
		$\frac{Z}{1+|Z|}$. 
		 Now we set $S_1:=\infty$ and we define
		inductively $S_n$, $n=2,3,\dots$, as a $\Lambda$-stopping time
		from the Meyer-Section Theorem
		\ref{meyer_thm_2}, which we apply for $\epsilon_n:=2^{-n}$ and
		\begin{align}\label{extension_eq_11}
			B_n:=\stsetO{T}{\infty}\cap \stsetRO{0}{\min\left(T+\frac1n,
			S_{n-1}\right)}\cap \left\{\left|  Z-Z^\ast_T\right|
			<\frac1n\right\}.
		\end{align}
		We just have to prove $B_n\in \Lambda$ which we will do below. Granted $B_n\in \Lambda$,
		we can define $T_1:=T+1$ and 
		$T_n:=\min(S_n,T_{n-1},T+1)$, $n=2,3,\dots$, and 
		the sequence $(T_n)_{n\in \mathbb{N}}\subset \stm$
		will satisfy the
		desired properties as $\pi(B_n)=\{T<\infty\}$.
		Indeed, a Borel-Cantelli argument applied to $\{S_n=\infty\}\cap\{T<\infty\}$, $n\in \NZ$,
		shows that for $\WM$-almost every
		$\omega\in \{T<\infty\}$ there exists $N_\omega$ such that for $n\geq N_\omega$ we have $S_n(\omega)<\infty$
		for $n\geq N_\omega$. In particular, $T(\omega)
		<S_n(\omega)<T(\omega)+\frac1n$,
		$|Z_{S_n}(\omega)-Z_{T}^\ast(\omega)|<\frac1n$ eventually. For $\omega \in \{T=\infty\}$,
		we get that $\omega\notin \pi(B_n)$ for $n\in \mathbb{N}$ and
		hence $S_n(\omega)=\infty$ for all $n\in \mathbb{N}$.
				
		\emph{Proof of $B_n\in \Lambda$:}
		We will argue that each of the three
		sets in the specification \eqref{extension_eq_11}  of $B_n$
		is contained in $\Lambda$. 
		First, we have $\stsetO{T}{\infty}\in \Lambda$ by \citing{EL80}{Theorem 2.1) and Corollary 1.1)}{503-504}, as
		$T$ is an $\aFA_+$-stopping time by \eqref{extension_eq_1}. Next, we see that for any $n\in \NZ$
		the stopping time $T+\frac1n$ is a predictable
		$\mathcal{F}$-stopping time, which implies by 
		$\mathcal{P}(\mathcal{F})\subset \Lambda$ that
		it is a $\Lambda$-stopping time. Analogously, $S_1$ is a $\Lambda$-stopping time. 
		Hence, for any $n\in \NZ$ we have $\min(T+\frac1n,S_{n-1})\in \stm$ and,
		as $\Lambda$ is a $\sigma$-field, we get
		\[
			\stsetRO{0}{\min\left(T+\frac1n,S_{n-1}\right)}\in \Lambda.
		\]		
		Finally, we see that $\lsr{Z}{T}\mathbb{1}_{\stsetO{T}{\infty}}$is 
		$\mathcal{F}$-predictable, because it is left-continuous and 
		$\mathcal{F}$-adapted.
		Hence, again by \eqref{extension_eq_1} it is a $\Lambda$-measurable
		process. As $Z$ is $\Lambda$-measurable also
		$Z\mathbb{1}_{\stsetO{T}{\infty}}$ is
		a $\Lambda$-measurable process.
		 Therefore, $(Z-\lsr{Z}{T})
		\mathbb{1}_{\stsetO{T}{\infty}}$ is $\Lambda$-measurable
		and, as a consequence,
		\[
			\stsetO{T}{\infty}\cap\left\{\left|Z
			-Z^\ast_T\right|<\frac1n\right\}\in \Lambda
		\]
		for every $n\in \mathbb{N}$.
				
		\emph{Proof of (ii):} We assume w.l.o.g. that $T(\omega)>0$ for all $\omega\in \Omega$ as we could replace $T$ by $T_{\{T>0\}}$ with $T=T_{\{T>0\}}$ almost surely.
		For $T$ there exists an announcing sequence
		$(\tilde{T}_n)_{n\in \mathbb{N}}$ by \citing{DM82}{Theorem 77}{132}, i.e
		a sequence of predictable $\mathcal{F}$-stopping times 
		$(\tilde{T}_n)_{n\in \mathbb{N}}$
		with $\tilde{T}_n(\omega)<T(\omega)$ and $\lim_{n\rightarrow \infty} \tilde{T}_n(\omega)=T(\omega)$ for all $\omega \in \Omega$. 
		Now we adapt the proof of $(i)$ with
		\[
			\tilde{B}_n:=\stsetO{\tilde{T}_n}{T}
			\cap \left\{\left|Z-{^\Lambda{(^\ast Z_T)}}\right|<\frac1n\right\},
		\]
		which satisfies $\WM(\pi(\tilde{B}_n))=1$,
		to obtain again a sequence of $\Lambda$-stopping times $S_n$ by the Meyer Section
		Theorem and the desired sequence is given by 
		$T_n:=\inf_{k\geq n} S_k$. Indeed, again a Borel-Cantelli
		argument applied to $\{S_n=\infty\}$, $n\in \NZ$, shows that for almost every
		$\omega\in \Omega$ there exists $N_\omega$ such that
		for $n\geq N_\omega$ we have $S_n(\omega)<\infty$.
		In particular $\tilde{T}_n(\omega)<S_n(\omega)<T(\omega)$ and
		$|Z_{S_n}(\omega)-{^\Lambda{(^\ast Z_T)}}(\omega)|
		<\frac1n$ for $n\geq N_\omega$.
		As $(\tilde{T}_n)_{n\in \mathbb{N}}$ is 
		non-decreasing to $T$, we see that 
		$T_n(\omega)=\inf_{k\geq n}S_k(\omega)$ is actually a minimum over finitely many
		elements. Hence, $T_n$ inherits the desired properties from $S_k, k\geq n$.
	\end{proof}
		

	\subsection{Path-properties of a process under the assumption 
	of upper-semicontinuity in expectation}	\label{extension_ssec_3}
		
		In this subsection we will extend Theorem II.1 of
		\cite{BS77} to the case of $\Lambda$-measurable processes.
				
		\begin{Lem}\label{extension_lem_2}
			Assume that $\Lambda$ is embedded in the
		sense of \eqref{extension_eq_1}. Let $Z$ be a
		$\Lambda$-measurable process of class($\text{D}^\Lambda$)
			with $Z_{\infty}=0$.
			Then the following holds:
			\begin{enumerate}[(i)]
				\item For any $S\in \stm$, the following assertions are equivalent:
					\begin{enumerate}
						\item $Z$ is right-upper-semicontinuous in expectation
								in $S$ in the sense that for 
								any non-increasing sequence $(S_n)_{n\in \mathbb{N}}\subset
								 \stm$ with $S_n\geq S$, $\infty>S_n>S$ on 
								$\{S<\infty \}$, $\lim_{n\rightarrow \infty}S_n=S$,
								$\lim_{n\rightarrow \infty} Z_{S_n}=Z^\ast_{S}$ and for any
								$A\in\mathcal{F}^\Lambda_S$ we have
								\[
									\mathbb{E}[Z_{S_A}]\geq \lim_{n\rightarrow \infty}
									\mathbb{E}[Z_{(S_n)_A}].
								\]
						\item $Z_S\geq {^\Lambda (\lsr{Z}{})_S}$ almost surely. 
					\end{enumerate}
					In particular, $Z$ is right-upper-semicontinuous in expectation in 
					all $S\in \stm$ if and only if 
					the set $\{Z<{^\Lambda (\lsr{Z}{})}\}$ is evanescent.					
				\item  For $S\in \stp$, the following assertions are equivalent:
					\begin{enumerate}
						\item For any non-decreasing sequence 
									$(S_n)_{n\in \mathbb{N}}
									\subset {\stm}$ 
									with $S_n< S$ on $\{S>0\}$, 
									$\lim_{n\rightarrow \infty} S_n=S$, 
									we have
									\begin{align}\label{extension_eq_2}
											\mathbb{E}\left[Z_{S}\right]
											\geq \limsup_{n\rightarrow \infty}
													\mathbb{E}\left[Z_{S_n}\right].
									\end{align}						
						\item For any non-decreasing sequence 
									$(S_n)_{n\in \mathbb{N}}
									\subset {\stm}$ 
									with $S_n< S$ on $\{S>0\}$, 
									$\lim_{n\rightarrow \infty} S_n=S$, 
									$\lim_{n\rightarrow
									\infty} Z_{S_n}={^\ast Z}_S$ and 				
									$A\in \bigcup_{m\in \mathbb{N}} \mathcal{F}_{S_m}^\Lambda$
									we have
									\begin{align}\label{extension_eq_3}
											\mathbb{E}\left[Z_{S}\mathbb{1}_{A}\right]
											\geq \lim_{n\rightarrow \infty}
													\mathbb{E}\left[Z_{S_n}\mathbb{1}_{A}\right].
									\end{align}						
						\item ${^\mathcal{P} Z_S}\geq \lsl{Z}_S$ almost surely.
					\end{enumerate}
					If the process $Z$
				satisfies one of the three equivalent conditions then we call $Z$
				left-upper-semicontinuous in expectation at $S$.
				
				In particular, the process $Z$ is left-upper-semicontinuous
				in expectation at every 
					$S\in \stp$ if and only if 
					the set $\{^\mathcal{P} Z< {^\ast Z}\}$ is evanescent and we have  
					${^\ast Z_\infty\leq {^\mathcal{P} Z}_\infty=0}$.
			\end{enumerate}
		\end{Lem}		
		
		\begin{Rem}[Right-upper-semicontinuity in expectation]
		\label{extension_rem_2}		
			\begin{enumerate}[(i)]
				\item For the class of optional processes, \cite{BS77}, p.306,
							have shown that even
							${^\mathcal{O} (\lsr{Z}{})_S}\geq \lsr{Z}{S}$ (see also Proposition \ref{extension_pro_3}). The analogous statement does not hold  true for general $\Lambda$-measurable process. Pathwise upper-semicontinuity of $\Lambda$-measurable processes thus has to be ensured by further assumptions (see \citing{Bes19}{Example 2.11}{47}).
				\item If $Z$ is right-upper-semicontinuous in expectation in 
						\emph{all} $S\in \stm$,
						the initial definition of upper-semicontinuity
						in expectation becomes
						easier to state. In fact, it amounts to the requirement
						that for all $S\in \stm$ we want that 
						for any non-increasing sequence $S_n\in 
						\stm$ with $S_n\geq S$, $S_n>S$
						on $\{S<\infty\}$, $\lim_{n\rightarrow \infty}S_n=S$, 
						$\lim_{n\rightarrow \infty} Z_{S_n}=Z^\ast_{S}$
					we have
					\[
						\mathbb{E}[Z_{S}]\geq \lim_{n\rightarrow \infty}
						\mathbb{E}[Z_{S_n}].
					\] 
					This follows by the fact that if $A\in \mathcal{F}^\Lambda_S
					\subset \mathcal{F}^\Lambda_{S_n}$ 
					also $S_A,(S_n)_A\in \stm$ and 
					$\lim_{n\rightarrow \infty}(S_n)_A=S_A$.
					\item The classical definition of right-upper-semicontinuity given by
					 \cite{DL82}, p.303, implies
					our definition and in the case of optional processes they
					are equivalent. Indeed, for optional
					processes \citing{DL82}{Remark 7 b)}{303}, implies that their  definition of 
					right-upper-semicontinuity implies $Z_S\geq {Z^\ast_S}$. Hence
					\begin{align*}
						Z_S=Z_S\mathbb{1}_{\{S<\infty\}}=\mathbb{E}\left[Z_S
						\mathbb{1}_{\{S<\infty\}}\middle|\mathcal{F}_S\right]
						&\geq \mathbb{E}\left[{Z^\ast_S}\mathbb{1}_{\{S<\infty\}}
						\middle|\mathcal{F}_S\right]\\
						&={^\mathcal{O} (Z^\ast)_S}\mathbb{1}_{\{S<\infty\}}
						={^\mathcal{O} (Z^\ast)_S},
					\end{align*}
					and this is equivalent to our definition of 
					right-upper-semicontinuity. 
					\item Analogously to Remark \ref{extension_rem_1} (ii), Lemma 
					\ref{extension_lem_2} also characterizes right- and
					left-lower-semicontinuity.
				\end{enumerate}
		\end{Rem}		
		
		\begin{proof}[Proof of Lemma \ref{extension_lem_2}]
			The proof of (i) is 
			mainly the proof in \citing{DL82}{Theorem 6}{303},
			and the proof of (ii) will be accomplished by adapting the same argument to
			the setting of predictable stopping times.
			
			\emph{Part (i):} a) ``$\Rightarrow$'' b):
			Assume $Z$ is right-upper-semicontinuous in expectation in $S$.
			By Proposition \ref{extension_pro_1} (i), there exists 
			 a non-increasing sequence $(S_n)_{n\in \mathbb{N}}\subset
			\stm$ such that $S_n\geq S$, 
			$\infty>S_n>S$ on $\{S<\infty\}$,
			$\lim_{n\rightarrow \infty} S_n=S$  and $\lim_{n\rightarrow \infty} 
			Z_{S_n}=\lsr{Z}{S}$ almost surely.		
			Let $A\in \mathcal{F}^\Lambda_S$. Then we get by $\lsr{Z}{\infty}=
			Z_{\infty}=0$ and because $Z$ is of class($D^\Lambda$) that
			\begin{align*}
				\mathbb{E}\left[Z_{S}\mathbb{1}_{\{S<\infty\}}\mathbb{1}_A\right]
				=\mathbb{E}\left[Z_{S_A}\right]
						&\geq
						 \lim_{n\rightarrow \infty} 
						\mathbb{E}\left[Z_{(S_n)_A}\right]\\
						&= \mathbb{E}\left[ \lim_{n\rightarrow \infty} 
						Z_{(S_n)_A}\right]
						= \mathbb{E}\left[ \lsr{Z}{S_A}\right]
						= \mathbb{E}\left[ \lsr{Z}{S}\mathbb{1}_{\{S<\infty\}}
						\mathbb{1}_A\right].
			\end{align*} 
			As $A\in \mathcal{F}^\Lambda_S$ is arbitrary this shows 
			that $Z_S\mathbb{1}_{\{S<\infty\}}\geq \mathbb{E}
			\left[Z_S^\ast\mathbb{1}_{\{S<\infty\}}\middle | 
			\mathcal{F}^\Lambda_S\right]$. Hence 
			\[
				Z_S\mathbb{1}_{\{S<\infty\}}\geq \mathbb{E}
			\left[Z_S^\ast\mathbb{1}_{\{S<\infty\}}\middle | 
			\mathcal{F}^\Lambda_S\right]
			={^\Lambda (Z^\ast)_S}\mathbb{1}_{\{S<\infty\}}
			\]
			and the rest follows by $Z_\infty={^\Lambda (Z^\ast)_\infty}=0$.
			
			b) ``$\Rightarrow$'' a): Assume now that we have ${Z_S}\geq {^\Lambda (Z^\ast)}_S$ 
			and let 
			 $(S_n)_{n\in \mathbb{N}}\subset
			\stm$ be a sequence 
							such that $S_n\geq S$, 
			$\infty>S_n>S$ on $\{S<\infty\}$,
			$\lim_{n\rightarrow \infty} S_n=S$  and $\lim_{n\rightarrow \infty} 
			Z_{S_n}=\lsr{Z}{S}$ almost surely. Then we obtain for 
			$A\in \mathcal{F}_{S}^\Lambda$
							that
							\begin{align*}
			\mathbb{E}\left[Z_{S_A}\right]=
			\mathbb{E}\left[Z_{S}\mathbb{1}_{\{S<\infty\}}\mathbb{1}_A\right]
			&\geq \mathbb{E}\left[{^\Lambda (Z^\ast)}_S
			\mathbb{1}_{\{S<\infty\}}\mathbb{1}_A\right]
			=  \mathbb{E}\left[Z^\ast_S\mathbb{1}_A\right]\\
						&= \mathbb{E}\left[ \lim_{n\rightarrow \infty} 
						Z_{(S_n)_A}\right]
						=\lim_{n\rightarrow \infty} \mathbb{E}\left[Z_{(S_n)_A}
						\right],
			\end{align*} 
			which finishes the proof of (i).
			
			\emph{Part (ii):} We can assume $S>0$ as we
			could replace $S$ by $S_{\{S>0\}}$ because ${^\mathcal{P} Z_0}=Z_0=\lsl{Z}_0$.
			
			\emph{a)\ $\Rightarrow$ b)}:
					Assume we have a non-decreasing sequence 
					$(S_n)_{n\in \mathbb{N}}
					\subset {\stm}$ 
					with $S_n< S$, 
					$\lim_{n\rightarrow \infty} S_n=S$, 
					$\lim_{n\rightarrow
					\infty} Z_{S_n}={^\ast Z}_S$ and 				
					$A\in \mathcal{F}_{S_m}^\Lambda$ for some $m\in \NZ$.
					Define $\tilde{S}_{n}:=S_{n+m}$ and observe that
					because $A\in \mathcal{F}_{\tilde{S}_n}^\Lambda$ also
					$((\tilde{S}_n)_A)_{n\in \mathbb{N}}\subset 
					\stm$ and this sequence satisfies the conditions
					of (a). Hence
					\[
						\mathbb{E}\left[Z_{S}\mathbb{1}_A\right]
						\geq \lim_{n\rightarrow \infty} 
						\mathbb{E}\left[Z_{\tilde{S}_n}\mathbb{1}_A\right]
						=\lim_{n\rightarrow \infty} 
						\mathbb{E}\left[Z_{S_n}\mathbb{1}_A\right],
					\]
					which we wanted to show.
			
			\emph{b)\ $\Rightarrow$ c):}
			By Proposition \ref{extension_pro_1} (ii),  there exists
			 a sequence $(S_n)_{n\in \mathbb{N}}\subset
			\stm$,
							$S_n< S$,
						$\lim_{n \rightarrow \infty}S_n=S$ and
							$\lsl{Z}_{S}=\lim_{n\rightarrow \infty} Z_{S_n}$ 
							almost surely.
			 Hence we get for $A\in \mathcal{F}^\Lambda_{S_m}$, 
			$m\in \mathbb{N}$,
			that
			\begin{align*}
				\mathbb{E}\left[Z_{S}\mathbb{1}_{\{S<\infty\}}\mathbb{1}_A
				\right]&=\mathbb{E}\left[Z_{S_A}\right]
						\geq \lim_{n\rightarrow \infty} \mathbb{E}
						\left[Z_{(S_n)_A}\right]\\
						&= \mathbb{E}\left[ \lim_{n\rightarrow \infty} 
						Z_{(S_n)_A}\right]
						= \mathbb{E}\left[ \lsl{Z}_{S_A}\right]
						= \mathbb{E}\left[ \lsl{Z}_{S}\mathbb{1}_A\right].
			\end{align*} 
			Since $A\in \mathcal{F}_{S_m}$ is arbitrary, this implies 
			that 
			\[
				\mathbb{E}
			\left[Z_S \mathbb{1}_{\{S<\infty\}}\middle | 
			\mathcal{F}_{S_m}\right]\geq 
			\mathbb{E}
			\left[{^\ast Z_S}\middle | 
			\mathcal{F}_{S_m}\right]
			\]
			 for all $m\in \mathbb{N}$. Thus, \citing{DM82}{Theorem 31}{26}, allows us to conclude for $m\rightarrow \infty$ that
			\begin{align*}
			\mathbb{E}
			\left[Z_S \mathbb{1}_{\{S<\infty\}}\middle | 
			\mathcal{F}_{S-}\right]&=\lim_{m\rightarrow \infty}
			\mathbb{E}
			\left[Z_S \mathbb{1}_{\{S<\infty\}}\middle | 
			\mathcal{F}_{S_m}\right]\\
			&\geq 
			\lim_{m\rightarrow \infty} \mathbb{E}
			\left[{^\ast Z_S}\middle | 
			\mathcal{F}_{S_m}\right]
			=\mathbb{E}
			\left[{^\ast Z_S}\middle | 
			\mathcal{F}_{S-}\right]={^\ast Z_S}
			\end{align*}
			 and, due to ${^\mathcal{P} Z_\infty}=0$, we get
			\[
				{^\mathcal{P} Z_S}={^\mathcal{P} Z_S}\mathbb{1}_{\{S<\infty\}}= \mathbb{E}\left[Z_S 
				\mathbb{1}_{\{S<\infty\}}\middle | 
			\mathcal{F}_{S-}\right]\geq 
			{^\ast Z_S},
			\]
			which we wanted to show.
			
			\emph{c)\ $\Rightarrow$ a):}
			Assume now that ${^\mathcal{P} Z_S}\geq {^\ast Z}_S$ holds
			and let 
			 $(S_n)_{n\in \mathbb{N}}\subset
			\stm$ be a non-decreasing 
			sequence such that
			$S_n< S$,
			$\lim_{n \rightarrow \infty}S_n=S$.
			Then we have 
			\begin{align*}
				\mathbb{E}\left[Z_{S}\right]=
				\mathbb{E}\left[Z_{S}\mathbb{1}_{\{S<\infty\}}\right]
				&=\mathbb{E}\left[{^\mathcal{P} Z}_S\mathbb{1}_{\{S<\infty\}}\right]
				=\mathbb{E}\left[{^\mathcal{P} Z}_S\right]\\
				&\geq  \mathbb{E}\left[ \lsl{Z}_{S}\right]
					\geq \mathbb{E}\left[ \limsup_{n\rightarrow\infty} 
														Z_{S_n}\right]\geq 
				\limsup_{n\rightarrow\infty} 	\mathbb{E}\left[ 
														Z_{S_n}\right],
			\end{align*} 
			which finishes our proof.							
		\end{proof}	
			
		In the next proposition we see that under some regularity
		conditions we can get more information
		about the sets $H_S^-,H_S$ and $H_S^+$ of Proposition
		\ref{optstop_pro_5} and we can indeed close 
		the small gap in El Karoui's proof of Theorem \ref{optstop_thm_6} (see Remark 
		\ref{optstop_rem_8}).
					
		\begin{Pro}\label{extension_pro_2}
			We use the notation from Proposition \ref{optstop_pro_4} and
			\ref{optstop_pro_5}. Let $Z$ be a positive
			$\Lambda$-measurable process of class($D^\Lambda$), which is 
			left-upper-semicontinuous in expectation at every
			$S\in \stp$ (see Lemma \ref{extension_lem_2} (ii)).
			Then we have for any fixed $S\in \stm$, that $H_S^-\in \mathcal{F}^\Lambda_{T_S-}$ and
			\[
				H_S^-\subset \{\bar{Z}_{T_S}=Z_{T_S}\}
			\]
			up to a $\mathbb{P}$-nullset. In particular, we get, up to $\mathbb{P}$-nullsets, that
			\[
				H_S^-\cup H_S=\{\bar{Z}_{T_S}=Z_{T_S}\},
				\quad H_S^+=\{\bar{Z}_{T_S}<Z_{T_S}\}.
			\]
			In the optional case $\Lambda=\mathcal{O}$ and if in addition
			 $Z$ is right-upper-semicontinuous in expectation at all stopping
			times (see Lemma \ref{extension_lem_2} (i)), then the stopping time
			\[
				\bar{T}_S:=\inf\left\{t\geq S\ \middle\vert\  
				\bar{Z}_t=Z_t\right\}.
			\]
			satisfies 
			\[
				Z_{\bar{T}_S}=\bar{Z}_{\bar{T}_S}\quad
				\text{ and } \quad
				\mathbb{E}[Z_{\bar{T}_S}|\aFA_S]=\bar{Z}_S.
			\]
			In particular, $\bar{T}_0$ is optimal for \eqref{optstop_eq_4}.
		\end{Pro}		
		
		\begin{proof}		
			First we get by Proposition 
			\ref{optstop_pro_5} that $R:=(T_S)_{H_S^-}$
			is a predictable $\mathcal{F}^\Lambda_+$-stopping time and,
			hence, by Lemma \ref{extension_lem_1}, it is a 
			predictable $\mathcal{F}$-stopping time and a $\Lambda$-stopping time. 
			 Furthermore, 
			  \citing{DM78}{Theorem 56}{118}, yields again with Lemma \ref{extension_lem_1}
			\[
				H_S^-=\bigcap_{n=1}^\infty 
				 \left\{T_S^{1-\frac1n}
				<T_S\right\}\in 
				\mathcal{F}_{T_S-}=\mathcal{F}_{T_S-}^\Lambda.
			\]
			Now we want to show that
				$H_S^-\subset \{\bar{Z}_{T_S}=Z_{T_S}\}$ for $S\in \stm$.
			By Lemma \ref{extension_lem_2} (ii),
			we know $^\mathcal{P} Z\geq \lsl{Z}$
			and, by \citing{DM82}{Remark (a)}{104},
			we have $^\mathcal{P} Z\leq {^\mathcal{P} \bar{Z}}$ since $Z\leq \bar{Z}$. Hence, 
			$\lsl{Z}\leq {^\mathcal{P} Z}\leq {^\mathcal{P} \bar{Z}}$ and,
			by Proposition \ref{optstop_pro_4}, this yields 
			$\bar{Z}_-=(^\mathcal{P} \bar{Z})\vee \lsl{Z}={^\mathcal{P} \bar{Z}}$.
			Combining this with the predictability of $R$ and
			 $\bar{Z}_{R-}=\lsl{Z}_R$ (Proposition \ref{optstop_pro_5}) 
			 gives us
			\begin{align}\label{extension_eq_4}
			\mathbb{E}\left[\bar{Z}_R\mathbb{1}_{\{R<\infty\}}\right]
				&\nonumber
				=\mathbb{E}\left[\mathbb{E}\left[\bar{Z}_R
				\mathbb{1}_{\{R<\infty\}}\ \middle\vert\  
				\mathcal{F}_{R-}\right]\right]
				=	\mathbb{E}\left[^\mathcal{P} \bar{Z}_R
				\mathbb{1}_{\{R<\infty\}}\right]\\
				&=	\mathbb{E}\left[\bar{Z}_{R-}
				\mathbb{1}_{\{R<\infty\}}\right]\nonumber
				=\mathbb{E}\left[\lsl{Z}_{R}\mathbb{1}_{\{R<\infty\}}\right]\\
				&\leq \mathbb{E}\left[^\mathcal{P} Z_R
				\mathbb{1}_{\{R<\infty\}}\right]\nonumber
				=\mathbb{E}\left[\mathbb{E}\left[Z_R
				\mathbb{1}_{\{R<\infty\}}\ \middle\vert\ 
				\mathcal{F}_{R-}\right]\right]\\
				&= \mathbb{E}\left[ Z_R\mathbb{1}_{\{R<\infty\}}\right].
			\end{align}
			By $\bar{Z}_{\infty}=Z_{\infty}=0$ we can conclude
			\begin{align*}
			\mathbb{E}\left[\bar{Z}_{T_S}\mathbb{1}_{H_S^-}\right]
					=\mathbb{E}\left[Z_R\mathbb{1}_{\{R=\infty\}}
					\right]
							+\mathbb{E}\left[\bar{Z}_R
							\mathbb{1}_{\{R<\infty\}}\right]
				\leq \mathbb{E}\left[ Z_{T_S}\mathbb{1}_{H_S^-}\right]
			\end{align*}
			and, because $Z\leq \bar{Z}$, we get
			\[
					\bar{Z}_{T_S}\mathbb{1}_{H_S^-}=Z_{T_S}\mathbb{1}_{H_S^-}.
			\]
			The ``in particular part'' follows because $H_S\subset \left\{
			\bar{Z}_{T_S}= Z_{T_S}\right\}$, 
			$H_S^+\subset \left\{\bar{Z}_{T_S}< Z_{T_S}\right\}$ 
			and $H_S^-\cup H_S\cup H_S^+=\Omega$ (see Proposition 
			\ref{optstop_pro_5}).			
			
			Finally we prove the assertion in the optional case. By
			 Proposition \ref{optstop_pro_7} we get that $Z$ is 
			pathwise upper-semicontinuous
			  from the right and by repeating the arguments in \eqref{extension_eq_4} we have with $Z_\infty=0$ that
			\begin{align}\label{extension_eq_5}
				\mathbb{E}\left[\lsl{Z}_{T_S}\mathbb{1}_{H_S^-}
							\mathbb{1}_{\{T_S<\infty\}}\midG \aFA_S\right]
							&= \mathbb{E}\left[Z_{T_S}\mathbb{1}_{H_S^-}
							\mathbb{1}_{\{T_S<\infty\}}\midG \aFA_S\right]\\
							&=\mathbb{E}\left[Z_{T_S}\mathbb{1}_{H_S^-}
							\midG \aFA_S\right].
			\end{align}
			Moreover, we get by Lemma 
			\ref{extension_lem_2} (ii) that ${^\ast Z_\infty} \leq 
			{^\mathcal{P} Z_\infty}=0$, which implies by \eqref{extension_eq_5}
			\begin{align}\label{Main:253}
				\mathbb{E}\left[\lsl{Z}_{T_S}\mathbb{1}_{H_S^-}
							\midG \aFA_S\right]
							\leq \mathbb{E}\left[Z_{T_S}\mathbb{1}_{H_S^-}
							\midG \aFA_S\right].
			\end{align}
			Combining this result with \eqref{Main:230},
			the pathwise upper-semicontinuity from the right of $Z$ and $Z_\infty^\ast=0$,
			 we obtain
			\begin{align*}
				\bar{Z}_S&\overset{\eqref{Main:230}}{=}\mathbb{E}\left[\lsl{Z}_{T_S}
				\mathbb{1}_{H_S^-}+Z_{T_S}\mathbb{1}_{H_S}
								+\lsr{Z}{T_S}\mathbb{1}_{H_S^+}\midG \aFA_S\right]\\
						&\leq \mathbb{E}\left[\lsl{Z}_{T_S}
						\mathbb{1}_{H_S^-}+Z_{T_S}\mathbb{1}_{H_S\cup H_S^+}
								\midG \aFA_S\right]	
						\overset{\eqref{Main:253}}{\leq} \mathbb{E}\left[Z_{T_S}| \aFA_S\right]
								\leq \mathbb{E}\left[\bar{Z}_{T_S}| \aFA_S\right]\leq
								\bar{Z}_S.		
			\end{align*}
			Here, we have used in the last step that $\bar{Z}$ is a 
			supermartingale and $T_S$ is a
			$\Lambda$-stopping time as $\Lambda=\mathcal{O}$.
			This shows $\mathbb{E}[Z_{T_S}| \aFA_S]= \mathbb{E}[\bar{Z}_{T_S}| \aFA_S]
			=\bar{Z}_{S}$, which in particular implies, by $Z\leq \bar{Z}$, that
			$Z_{T_S} = \bar{Z}_{T_S}$ almost surely and hence $\bar{T}_S=T_S$.			
		\end{proof}		


	\subsection{Limit results for Meyer projections}\label{extension_ssec_4}
	
		Assume that $\Lambda$ is
		embedded in the
		sense of \eqref{extension_eq_1}.
		The next result can be viewed as a version of Fatou's lemma for Meyer projections.
	
		\begin{Pro}\label{extension_pro_3}
			Assume that $\Lambda$ is
					embedded in the
					sense of \eqref{extension_eq_1}. Let $Z$ be an $(\mathcal{F}\otimes 
					\mathcal{B}([0,\infty)))$-measurable process 
					of class($D^\Lambda$) with $Z_\infty=0$. Then the following two assertions hold true:
			\begin{enumerate}[(i)]
				\item For an arbitrary $\mathcal{F}$-stopping time $T$, we have		
					\[
						{^{\mathcal{O}}
						(Z_\ast)}_T\leq \left({^\Lambda Z}\right)_{T\ast}
						\leq \left({^\Lambda Z}\right)_{T}^\ast\leq {^{\mathcal{O}}
						(Z^\ast)}_T,
					\]
					where
					${\mathcal{O}}$ is the optional-$\sigma$-field
					with respect to the filtration $(\mathcal{F}_t)_{t\geq 0}$.
					In particular, if $Z$ has right-limits then we get
					\[
						\left({^\Lambda Z}\right)_{T+}={^{\mathcal{O}}
						(Z_+)}_T.
					\]
				\item For an arbitrary predictable 
				$\mathcal{F}$-stopping time $T$, we have
					\[
						{^{\mathcal{P}}
						({_\ast Z})}_T\leq {_\ast \left({^\Lambda Z}\right)_{T}}
						\leq {^\ast \left({^\Lambda Z}\right)_{T}}\leq {^{\mathcal{P}}
						({^\ast Z})}_T,
					\]
					where
					${\mathcal{P}}$ denotes the predictable-$\sigma$-field
					with respect to the filtration $(\mathcal{F}_t)_{t\geq 0}$.
					In particular, if $Z$ has left-limits then we get
					\[
						{\left({^\Lambda Z}\right)_{T-}}={^{\mathcal{P}}
						({Z_-})}_T.
					\]
			\end{enumerate}
		\end{Pro}
		
		\begin{Rem}\label{extension_rem_3}
			 One should remark that the above results for example do not imply that $\{ \left({^\Lambda Z}\right)^\ast> {^{\mathcal{O}}
						(Z^\ast)}\}$ is an evanescent set as
				one can not apply in general the Meyer section 
				theorem to $\left({^\Lambda Z}\right)^\ast$.
			\end{Rem}
		
		\begin{proof}
			\emph{Part (i):}
			Let $T$ be an arbitrary $\mathcal{F}$-stopping time.
			Then there exists by Proposition \ref{extension_pro_1} (i)
			a sequence $(T_n)_{n\in \mathbb{N}}\subset \stm$
			such that $T_n\geq T$, $T<T_n<\infty$ on $\{T<\infty\}$,
			$\lim_{n\rightarrow \infty} T_n=T$ and $(^\Lambda Z)_{T}^\ast=
			\lim_{n\rightarrow \infty} (^\Lambda Z)_{T_n}$. Hence, we get on
			$\{T<\infty\}$ 
			\begin{align}\label{extension_eq_6}
				(^\Lambda Z)_{T}^\ast=\lim_{n\rightarrow \infty} (^\Lambda Z)_{T_n}
				=\lim_{n\rightarrow \infty} \mathbb{E}\left[Z_{T_n}\middle|
				\mathcal{F}^\Lambda_{T_n}\right],
			\end{align}
			where $({^\Lambda Z})_T^\ast$ is $\FA_T$-measurable. Indeed, by \citing{DM78}{Theorem 90}{143} the process $({^\Lambda Z})^\ast$ is $\FA$-progressively measurable and therefore by \citing{DM78}{Theorem 64}{122}
			we have $({^\Lambda Z})_T^\ast$ is $\FA_T$-measurable.
			On the other hand we have on $\{T<\infty\}$ that
			\begin{align}\label{extension_eq_7}
				^\mathcal{O}(Z^\ast)_{T}=\mathbb{E}\left[Z_{T}^\ast
				\middle|
				\mathcal{F}_{T}\right]= \mathbb{E}\left[Z_{T}^\ast\middle|
				\mathcal{F}^\Lambda_{T+}\right],
			\end{align}
			where we have used that $\mathcal{F}_{T}=\mathcal{F}^\Lambda_{T+}$
			by Lemma \ref{extension_lem_1}.			
			For $A \in \mathcal{F}_T=\mathcal{F}^\Lambda_{T+}\subset
			\mathcal{F}^\Lambda_{T_n}$, we can now use Fatou's Lemma and that $Z$ is of class($D^\Lambda$) to conclude
			\begin{align}\label{extension_eq_8}
			\mathbb{E}\left[ \lim_{n\rightarrow \infty}
				\mathbb{E}\left[Z_{T_n}\middle|
				\mathcal{F}^\Lambda_{T_n}\right]
				\mathbb{1}_{A}\right]\nonumber
				&=\lim_{n\rightarrow \infty} \mathbb{E}\left[ 
				\mathbb{E}\left[Z_{T_n}\middle|
				\mathcal{F}^\Lambda_{T_n}\right]
				\mathbb{1}_{A}\right]\nonumber\\
				&=\lim_{n\rightarrow \infty} \mathbb{E}\left[ { Z}_{T_n}
				\mathbb{1}_{A}\right]
				\leq\mathbb{E}\left[ { Z}_{T}^\ast
				\mathbb{1}_{A}\right].
			\end{align}
			As $A \in \mathcal{F}_T=\mathcal{F}^\Lambda_{T+}$
			was arbitrary, equation
			\eqref{extension_eq_8} leads to
			\[
				 \lim_{n\rightarrow \infty}
				\mathbb{E}\left[Z_{T_n}\middle|
				\mathcal{F}^\Lambda_{T_n}\right]
				\leq \mathbb{E}\left[Z_{T}^\ast\middle|
				\mathcal{F}^\Lambda_{T+}\right],
			\]
			which proves $\left({^\Lambda Z}\right)_{T}^\ast\leq {^{\mathcal{O}}
			(Z^\ast)}_T$ by \eqref{extension_eq_6} and \eqref{extension_eq_7}.
			Analogously we get that ${^{\mathcal{O}}
			(Z_\ast)}_T\leq \left({^\Lambda Z}\right)_{T\ast}$ by using
			Fatou's Lemma into the other direction and Remark 
			\ref{extension_rem_1} (ii).
						
			\emph{Part (ii):}
			Let $T$ be an arbitrary predictable $\mathcal{F}$-stopping time, where we assume $T>0$ as we could replace $T$ by $T_{\{T>0\}}$ as we defined ${^\ast Z_0}={_\ast Z_0}=0$.
			Then there exists by Proposition \ref{extension_pro_1}
			a sequence $(T_n)_{n\in \mathbb{N}}\subset \stm$
			such that $T_n< T$, 
			$\lim_{n\rightarrow \infty} T_n=T$ and ${^\ast (^\Lambda Z)}_{T}=
			\lim_{n\rightarrow \infty} (^\Lambda Z)_{T_n}$. Hence, we get 
			\begin{align}\label{extension_eq_9}
				^\ast (^\Lambda Z)_{T}=\lim_{n\rightarrow \infty} 
				(^\Lambda Z)_{T_n}
				=\lim_{n\rightarrow \infty} \mathbb{E}\left[Z_{T_n}\middle|
				\mathcal{F}^\Lambda_{T_n}\right].
			\end{align}
			For $A \in \mathcal{F}_{T_m}^\Lambda$,
			with $m\in \mathbb{N}$ fixed, we can use Fatou's Lemma and that $Z$ is 
			of class($D^\Lambda$) to obtain
			\begin{align}\label{extension_eq_10}
				\mathbb{E}\left[ {^\ast({^\Lambda Z})_T}
				\mathbb{1}_{A}\right]&=\mathbb{E}\left[ \lim_{n\rightarrow \infty}
				\mathbb{E}\left[Z_{T_n}\middle|
				\mathcal{F}^\Lambda_{T_n}\right]
				\mathbb{1}_{A}\right]\nonumber\\
				&=		\lim_{n\rightarrow \infty} \mathbb{E}\left[ 
				\mathbb{E}\left[Z_{T_n}\middle|
				\mathcal{F}^\Lambda_{T_n}\right]
				\mathbb{1}_{A}\right]\nonumber\\		
				&=\lim_{n\rightarrow \infty} \mathbb{E}\left[ Z_{T_n}
				\mathbb{1}_{A}\right]
				\leq \mathbb{E}\left[ {^\ast Z}_{T}
				\mathbb{1}_{A}\right].
			\end{align}
			As $A \in \mathcal{F}_{T_m}$ was arbitrary, equation
			\eqref{extension_eq_10} leads to
			\[
				\mathbb{E}\left[{^\ast({^\Lambda Z})_T}\middle|
				\mathcal{F}^\Lambda_{T_m}\right]
				\leq \mathbb{E}\left[{^\ast Z}_T \middle|
				\mathcal{F}^\Lambda_{T_m}\right].
			\]
			Now we get by \citing{DM82}{Theorem 31}{26},
			for $m\rightarrow \infty$ and 
			$\mathcal{F}^\Lambda_{T-}=	\mathcal{F}_{T-}$ 
			(see Lemma \ref{extension_lem_1}) that
			\[
				{^\ast({^\Lambda Z})_T}=\lim_{m\rightarrow \infty}
				\mathbb{E}\left[{^\ast({^\Lambda Z})_T}\middle|
				\mathcal{F}^\Lambda_{T_m}\right]
				\leq \lim_{m\rightarrow \infty} 
				\mathbb{E}\left[{^\ast Z}_T \middle|
				\mathcal{F}^\Lambda_{T_m}\right]
				={^\mathcal{P} (^\ast Z)}_T,
			\]
			which proves the first part of our claim. Here we have used for the first identity that by 
			\citing{DM78}{Theorem 89}{143}, ${^\ast({^\Lambda Z})}$ is 
			predictable and hence
			${^\ast({^\Lambda Z})}_T$ is $\mathcal{F}_{T-}$-measurable.
			Analogously, we get 
			the other direction by using Fatou's Lemma for the limes inferior 
			and by recalling Remark \ref{extension_rem_1} (ii). 
		\end{proof}
		
			Let us give another limit theorem for Meyer projections and show that for a continuously parametrized family of stochastic processes, one can choose the corresponding Meyer projections so that they inherit this continuity.

        	\begin{Lem}\label{sto:tools_lem_2}
        	    Let $(Z_t^\ell)_{t\geq 0}$, $\ell \in \mathbb{R}$, be a family of processes such that for $\WM$-almost every $\omega\in \Omega$ we have locally uniform continuity with respect to $\ell$ in the sense that 
        	    \begin{align}\label{Main:3}
        	        	\lim_{\delta \downarrow 0}
        			\sup_{\overset{\ell,\ell'\in C}{|\ell'-\ell|
        			\leq \delta}}\sup_{t\in [0,\infty]}
        			 \left|Z^{\ell'}_t(\omega)-Z^{\ell}_t(\omega)\right|
        			=0\quad \text{for any compact $C\subset \mathbb{R}$.}
        		\end{align}
        		Assume in addition that, for any such $C$,
        		\begin{align}
        			\mathbb{E}\left[\sup_{\overset{\ell,\ell'\in C}{|\ell'-\ell|
        			\leq \delta}}\sup_{t\in [0,\infty]}        			 \left|Z^{\ell'}_t(\omega)-Z^{\ell}_t(\omega)\right|\right]<\infty.\label{Main:4}
        	    \end{align}
        	    Then the family  of $\Lambda$-projections ${^\Lambda (Z^\ell_t)_{t \geq 0}}$, $\ell\in \mathbb{R}$, can be chosen such that for any $\omega \in \Omega$ we also have local uniform continuity with respect to $\ell\in \mathbb{R}$ for also for the family of $\Lambda$-projections, i.e.
        		\[
        			\lim_{\delta \downarrow 0}
        			\sup_{\overset{\ell,\ell'\in C}{|\ell'-\ell|
        			\leq \delta}}\sup_{t\in [0,\infty]}
        			 \left|{^\Lambda}{}{\mathop{(Z^{\ell'})}}_t(\omega)-{^\Lambda}{}{\mathop{(Z^{\ell})}}_t(\omega)\right|
        			=0 \quad \text{for any compact $C\subset \mathbb{R}$}.
        		\]
        	\end{Lem}	
	    
	    	\begin{proof}
	    	Noting the arguments of \cite{KP17} only require Meyer's section and projection theorem, which also hold for Meyer-$\sigma$-fields, we first obtain that the $\Lambda$-projections of $Z^\ell$, $\ell \in \mathbb{R}$,
    		can be chosen such that
	    	\begin{align}\label{continuity}
    			\lim_{\ell'\rightarrow \ell} {^\Lambda}{}{\mathop{(Z^{\ell'})}}_t(\omega)
    			={^\Lambda}{}{\mathop{(Z^{\ell})}}_t(\omega)\quad \text{ for all } \ell\in \mathbb{R},
    			\ (\omega,t)\in \Omega\times [0,\infty]. 
    		\end{align}
    		To deduce uniform convergence as claimed, fix in the following a compact interval $I\subset \mathbb{R}$ and
    	    observe that
    		\[
    			Z(\delta):=\sup_{\overset{\ell,\ell'\in I}{|\ell'-\ell|
    			\leq \delta}}\sup_{t\in [0,\infty]} \left|Z_t^{\ell'}-Z_t^{\ell}\right|, \quad \delta>0,
    		\]
    		converges to zero almost surely and in $\mathrm{L}^1(\WM)$ as $\delta \downarrow 0$ because of~\eqref{Main:3} and \eqref{Main:4}.
    		Furthermore for fixed $\ell,\ell'\in I$ with $|\ell-\ell'|\leq \delta$ 
    		and any $T\in \stm$ we have 
    		\begin{align*}
    			\left|{^\Lambda}{}{\mathop{(Z^{\ell'})}}_T
    							-{^\Lambda}{}{\mathop{(Z^{\ell})}}_T\right|
    				&\leq \mathbb{E}\left[|Z^{\ell'}_T-Z^\ell_T|\middle| \aFA_T\right]\\
    				&\leq \mathbb{E}\left[Z(\delta)\middle| \aFA_T\right]
    				={^\Lambda}{}{\mathop{Z(\delta)}}_T\,\text{ a.s. on } \{T<\infty\}.
    		\end{align*}
    		Hence, by the Meyer Section Theorem, $|{^\Lambda}{}{\mathop{(Z^{\ell'})}}
    							-{^\Lambda}{}{\mathop{(Z^{\ell})}}|
    							\leq {^\Lambda}{}{Z(\delta)}$
    		up to an evanescent set for any fixed $\ell,\ell'\in I$ with $|\ell-\ell'|\leq \delta$.
    		Therefore, almost surely, 
    		\[
    			\sup_{\overset{\ell,\ell'\in I\cap \mathbb{Q}}{|\ell'-\ell|
    		\leq \delta}} \sup_{t\in [0,\infty]} |{^\Lambda}{}{\mathop{(Z^{\ell'})}}_t
    							-{^\Lambda}{}{\mathop{(Z^{\ell})}}_t|
    				\leq \sup_{t\in [0,\infty]} {^\Lambda}{}{Z(\delta)}_t.
    		\]
    		Interchanging its two suprema, the left-hand side can be rewritten as 
    		\begin{align*}
    			\sup_{\overset{\ell,\ell'\in I\cap \mathbb{Q}}{|\ell'-\ell|
    		\leq \delta}} \sup_{t\in [0,\infty]} |{^\Lambda}{}{\mathop{(Z^{\ell'})}}_t
    							-{^\Lambda}{}{\mathop{(Z^{\ell})}}_t|
    				&=\sup_{t\in [0,\infty]} \sup_{\overset{\ell,\ell'\in I\cap \mathbb{Q}}{|\ell'-\ell|
    		\leq \delta}}  |{^\Lambda}{}{\mathop{(Z^{\ell'})}}_t
    							-{^\Lambda}{}{\mathop{(Z^{\ell})}}_t|\\
    				&= \sup_{t\in [0,\infty]} \sup_{\overset{\ell,\ell'\in I}{|\ell'-\ell|
    		\leq \delta}} |{^\Lambda}{}{\mathop{(Z^{\ell'})}}_t
    							-{^\Lambda}{}{\mathop{(Z^{\ell})}}_t|,
    		\end{align*}
    		where we used the pointwise continuity \eqref{continuity} in the last equality.
    		Therefore we have almost surely
    		\[
    		 \sup_{\overset{\ell,\ell'\in I}{|\ell'-\ell|
    		\leq \delta}} 	\sup_{t\in [0,\infty]}|{^\Lambda}{}{\mathop{(Z^{\ell'})}}_t
    							-{^\Lambda}{}{\mathop{(Z^{\ell})}}_t|
    					\leq \sup_{t\in [0,\infty]} {^\Lambda}{}{Z(\delta)}_t,
    		\]
    		and now it suffices to argue that ${^\Lambda}{}{Z(\delta)}_t\rightarrow 0$ almost surely uniformly in $t\in [0,\infty]$.
    		For this, we use Doob's maximal martingale inequality, suitably generalized for $\Lambda$-martingales
    		(\citing{DM82}{Appendix 1, (3.1)}{394}). More precisely, for any $\lambda \in [0,\infty)$ we have by dominated convergence that
    		\[
    			\lambda \WM\left(\sup_{t\in [0,\infty]} {^\Lambda}{}{Z(\delta)}_t> \lambda\right)\leq \mathbb{E}\left[{^\Lambda Z}(\delta)_\infty\right]
    			=\mathbb{E}\left[Z(\delta)\right]\overset{\delta\downarrow 0}{\longrightarrow} 0,
    		\]
    		which finishes our proof.
    	\end{proof}
		
	\section{A stochastic representation theorem and universal stopping signals}\label{sec:rep}
        	In this section we want to study an optimal stopping problem over divided stopping times, where we try to find the optimal time to collect a terminal reward given by a $\Lambda$-measurable $(X_t)_{t\in [0,\infty)}$ when before one receives ``running rewards'' represented by some function $g(\ell)$. Specifically, we want to solve for any $\ell\in \mathbb{R}$ the optimal stopping problem with value
			\begin{align}\label{eq:0}
				\sup_{\tau\in \stmd}
				\mathbb{E}\left[X_\tau+\int_{[0,\tau)}g_t(\ell)\mu(\mathrm{d} t)\right].
			\end{align}
			To make this precise, let us next make the following assumptions.
		    As in the previous chapters, $\Lambda$ is a $\WM$-complete Meyer-$\sigma$-field
		    and we denote by $\FA:=(\FA_t)_{t\geq 0}$ a filtration satisfying the usual conditions such that $\mathcal{P}(\FA)\subset \Lambda\subset \mathcal{O}(\FA)$ (see Theorem \ref{meyer_thm_4}). The $\mu$ on $[0,\infty)$ and the random field $g:\Omega\times [0,\infty)\times \mathbb{R}\rightarrow \mathbb{R}$  satisfy the following assumption:
		    \begin{Ass}\label{sto:frame_ass}
				\begin{compactenum}[(i)]				
    				\item 	$\mu$ is a random Borel-measure on $[0,\infty)$ with $\mu(\{\infty\}):=0$.
					\item 	The random field $g:\Omega\times [0,\infty)\times \mathbb{R} \rightarrow \mathbb{R}$ satisfies:
					\begin{compactenum}
						\item 	For each $\omega\in \Omega$, $t\in [0,\infty)$, the 
							 	function $g_t(\omega,\cdot):\mathbb{R}\rightarrow \mathbb{R}$ is continuous and strictly increasing from $-\infty$ to $\infty$. \label{sto:ass_bullet_1}
						\item For each $\ell \in \mathbb{R}$, the \label{sto:ass_bullet_2}	process $g_\cdot(\cdot,\ell): \Omega\times [0,\infty) \rightarrow \mathbb{R}$ is $\FA\otimes\mathcal{B}([0,\infty))$-measurable with
					 	\[
					 		\mathbb{E}\left[\int_{[0,\infty)} |g_t(\ell)|\mu
							(\mathrm{d} t)\right]<\infty.
					 	\]
					\end{compactenum}
				\end{compactenum}
			\end{Ass}			
		    Moreover, the $\Lambda$-measurable process $X$ is chosen in such a way that there exists a $\Lambda$-measurable process $L$ satisfying 
		    \begin{align}\label{sto:frame_gl}
							X_S=\mathbb{E}\left[\int_{[S,\infty)} 
								g_t\left(\sup_{v\in [S,t]} L_v\right)
							 \mu(\mathrm{d} t)\midG \aFA_S\right]
							,\ S\in \stm,\\
							\mathbb{E}\left[\int_{[S,\infty)} \left\lvert 
								g_t\left(\sup_{v\in [S,t]} L_v\right)
							 \right\rvert \mu(\mathrm{d} t)\right]<\infty\;
							\text{ for any $S\in \stm$}.
							 \label{sto:frame_gl_2}
					\end{align}
					
		    In the representation theorem of \cite{BB18} one can find (rather mild) sufficient conditions on $X$ such that the previous representation is possible.
		    The optimal stopping problem \eqref{eq:0} in the optional case, is discussed in \citing{BF03}{Theorem 2}{6}, albeit only for atomless $\mu$ with full support. In this case a sufficient condition for a representation as in \eqref{sto:frame_gl} was proven in \cite{BK04}. Without these regularity properties, optimal stopping times can no longer be expected from a representation as in \eqref{sto:frame_gl}. But we still can describe optimal divided stopping times in terms of the representing process $L$ and thus provide an optimal stopping characterization alternative to the Snell envelope approach of Theorem \ref{optstop_thm_3}:
		    
		    \begin{Thm}\label{thm:opt_stop}
                Let $\Lambda$ be a $\WM$-complete Meyer-$\sigma$-field, $g$ and $\mu$ satisfy Assumption \ref{sto:frame_ass}. Suppose furthermore that $X$ is a $\Lambda$-measurable process of class($D^\Lambda$), which is left-upper-semicontinuous in expectation at every $S\in \stp$ (see Lemma \ref{extension_lem_2} (ii)), and that $X$ allows a representation by a $\Lambda$-measurable process $L$ with
			\eqref{sto:frame_gl} and \eqref{sto:frame_gl_2}. 
			
			Then $X$ is right-upper-semicontinuous in expectation at every $S\in \stm$ (see Lemma \ref{extension_lem_2} (i)) and $L$ is a universal stopping signal for~\eqref{eq:0} in the sense that, for any $\ell \in \mathbb{R}$, the divided stopping times
		    			\begin{align}\label{Main:10}
		    					\tau_\ell^{(i)}:=(T_{\ell}^{(i)},\emptyset,H_\ell^{(i)},
		    				(H_\ell^{(i)})^c),\quad i=1,2,
		    			\end{align}
		    			with 
		    			\begin{align*}
		    				H_\ell^{(1)}&:=\{L_{T_{\ell}^{(1)}}\geq  \ell\}\supset \{L_{T_{\ell}^{(2)}}>  \ell\}=:H_\ell^{(2)},\\
		    				T_\ell^{(1)}&:=\inf\left\{t\geq 0\midG \sup_{v\in [0,t]}L_v
		    				\geq \ell\right\}
								\leq \inf\left\{t\geq 0\midG \sup_{v\in [0,t]}L_v 
		    				> \ell\right\}=:T_\ell^{(2)}
		    			\end{align*}
		    			attain the supremum in \eqref{eq:0}.
            \end{Thm}
            
            \begin{proof}
                We first show that $X$ is right-upper-semicontinuous in expectation. For that we will adapt the proof of \citing{BK04}{Theorem 2}{1048}, to 
		our stochastic setting.
		 Define for arbitrary  $S\in \stm$ the process $i^S:\Omega\times [0,\infty)
		\rightarrow [0,\infty]$ by
			\begin{align}\label{Gl:c20}
				i^S_t(\omega)
				:=\mathbb{1}_{\stsetRO{S}{\infty}}(\omega,t)
							g_t\left(\omega,\sup_{v\in [S(\omega),t]} L_v(\omega)\right) \leq 0\vee g_t\left(\omega,\sup_{v\in [0,t]} L_v(\omega)\right).
			\end{align}
			Consider now a sequence
											$(S_n)_{n\in \NZ}\subset \stmg{S}$ with
											$\lim_{n\rightarrow \infty} \mu([S,S_n))=0$ almost surely. 
			Combining the estimate \eqref{Gl:c20} with~\eqref{sto:frame_gl_2}, we may use Fatou's lemma
			 to obtain from the representation property of $L$  that 
			\begin{align*}
				\limsup_{n\rightarrow \infty} \mathbb{E}\left[X_{S_n}\right] 
				=&\limsup_{n\rightarrow \infty} \mathbb{E}\left[\int_{[0,\infty)}  i^{S_n}_t
								\mu(\mathrm{d} t)\right] \\
				\leq  &\ \mathbb{E}\left[\limsup_{n\rightarrow \infty}  
				\int_{[0,\infty)} g_t\left(\sup_{v\in\, [S,t]} L_v\right)\mathbb{1}_{[S_n,\infty)}(t)
							\mu(\mathrm{d} t)\right]\\
				=& \ \mathbb{E}\left[\int_{[S,\infty)} g_t\left(\sup_{v\in [S,t]} L_v\right)
							\mu(\mathrm{d} t)\right]
			=  \mathbb{E}\left[X_S\right].
			\end{align*}
			which we wanted to show.

               The next arguments are inspired by \citing{BF03}{Theorem 2}{6}.
		               As a first step we get by Lemma \ref{extension_lem_2} (ii), that left-upper-semi-continuity in expectation is equivalent to $\lsl{X}\leq {^\mathcal{P}X}$ up to an evanescent set. Therefore, for any $\tau=(T,H^-,H,H^+)\in \stmd$ the alternative divided stopping time $\tilde{\tau}:=(T,\emptyset,H^-\cup H,H^+)$ yields at least as high a value in~\eqref{eq:0} as $\tau$ does. Here, $\tilde{\tau}$ is indeed a divided stopping time as by $\WM$-completeness of $\Lambda$ and Theorem \ref{meyer_thm_4} there exists a filtration $\FA:=(\FA_t)_{t\geq 0}$ satisfying the usual conditions such that
				\begin{align}\label{Main:1213}
						\mathcal{P}(\FA)\subset \Lambda \subset \mathcal{O}(\FA).
				\end{align}
By Lemma \ref{extension_lem_1} we then have $\mathcal{P}(\aFA_+)
=\mathcal{P}(\FA)\subset \Lambda$, which implies that the predictable 
$\aFA_+$-stopping time $T_{H^-}$ is also a $\Lambda$-stopping time. This shows that $\tilde{\tau}\in \stmd$ as the other parts of the definition of a divided stopping times are satisfied by having $\tau\in \stmd$.			
									Hence, we can confine ourselves to considering divided stopping times with $H^-=\emptyset$. 
		               
		               Next, we show that $\tau_\ell^{(1)}\in \stmd$. To this end note that by $\{\sup_{v\in [0,\cdot)}L_v<\ell\}\in \mathcal{P}(\aFA_+)\subset \Lambda$ the set $A_\ell:=\{L\geq \ell\}\cap \{\sup_{v\in [0,\cdot)}L_v<\ell\}$ is $\Lambda$-measurable and satisfies for $\omega\in \Omega$ that
		               \[
		               	(T_\ell)_{\{L_{T_\ell^{(1)}}\geq \ell\}}(\omega)=\inf\{t\geq 0|(\omega,t)\in A_\ell\}.
		               \]
										From \citing{DM78}{Theorem 50}{116}, we obtain that $T_\ell^{(1)}$ is an $\FA$-stopping time and by \citing{EL80}{Theorem 2}{503}, it is a stability time. 					
		                As $A_\ell$ contains the graph of $(T_\ell^{(1)})_{\{L_{T_\ell^{(1)}}\geq \ell\}}$
		               we get by 
		               an application of \citing{EL80}{Corollary 2}{504}, that $(T_\ell^{(1)})_{\{L_{T_\ell^{(1)}}\geq \ell\}}$ is a $\Lambda$-stopping time. It follows that $\tau_\ell^{(1)}$ is indeed a divided stopping time. One proves analogously  that $\tau^{(2)}_\ell \in \stmd$.
		               
		               Now it remains to prove that $\tau_\ell^{(i)}$, $i=1,2$, attains the supremum in~\eqref{eq:0}. Let us first argue why the representation property \eqref{sto:frame_gl} allows us to conclude that for any $\tau=(T,\emptyset,H,H^c) \in \stmd$ we have
		               \begin{align}\label{Main:0}
		                   \mathbb{E}\left[X_\tau+\int_{[0,\tau)}g_t(\ell)\mu(\mathrm{d} t)\right]
		                   \leq &\ \mathbb{E}\left[\mathbb{1}_H\int_{[T,\infty)}\left(g_t\left(\sup_{v\in [T,t]}L_v\right)-g_t(\ell)\right)\mu(\mathrm{d} t)\right. \\
		                   &+ \left.\mathbb{1}_{H^c}\int_{(T,\infty)}\left(g_t\left(\sup_{v\in (T,t]}L_v\right)-g_t(\ell)\right)\mu(\mathrm{d} t)\right] +C
		               \end{align}
		               with $C:= \mathbb{E}\left[\int_{[0,\infty)}g_t(\ell)\mu(\mathrm{d} t)\right]$. Indeed, first we obtain by \eqref{sto:frame_gl} and 
		               because $H\in \aFA_T$ with $T_H\in \stm$ that
		               \begin{align}
		                    \mathbb{E}\left[\mathbb{1}_H X_T\right]=\mathbb{E}\left[\mathbb{1}_H\int_{[T,\infty)} g_t\left(\sup_{v\in [T,t]}L_v\right)\mu(\mathrm{d} t)\right].
		               \end{align}
		               Second, we obtain by Proposition \ref{extension_pro_1} (i), that there exists a sequence of $\Lambda$-stopping times
		               $(T_n)_{n\in \NZ}$ such that $T_n\geq T_{H^c}$, $\infty>T_n>T_{H^c}$ on $\{T_{H^c}<\infty\}$, $\lim_{n\rightarrow \infty} T_n=T_{H^c}$ and $X^\ast_{T_{H^c}}          =\lim_{n\rightarrow \infty} X_{T_n}$. Hence, again by \eqref{sto:frame_gl} and using that $X$ is of class($D^\Lambda$) with $X_\infty=X^\ast_\infty=0$, we get by Fatou's lemma that
		                \begin{align}\label{Main:2}
		                    \mathbb{E}\left[\mathbb{1}_{H^c} X_T^\ast\right]=\lim_{n\rightarrow\infty}\mathbb{E}[\mathbb{1}_{H^c}X_{T_n}]
		                    \leq \mathbb{E}\left[\mathbb{1}_{H^c}\int_{(T,\infty)} g_t\left(\sup_{v\in (T,t]}L_v\right)\mu(\mathrm{d} t)\right],
		               \end{align}
		               which finishes the proof of \eqref{Main:0}. 
		               
		               Choose $i\in \{1,2\}$ arbitrary. Now, one can use the monotonicity of $g$ to check that our divided level passage time $\tau_\ell^{(i)}$ of \eqref{Main:10} satisfies
		                \begin{align}
		                  &\mathbb{1}_H\int_{[T,\infty)}\left(g_t\left(\sup_{v\in [T,t]}L_v\right)-g_t(\ell)\right)\mu(\mathrm{d} t)\\&\quad+\mathbb{1}_{H^c}
		                  \int_{(T,\infty)}\left(g_t\left(\sup_{v\in (T,t]}L_v\right)-g_t(\ell)\right)\mu(\mathrm{d} t)\\
		                  &\leq \mathbb{1}_H\int_{[T,\infty)}\left(g_t\left(\sup_{v\in [T,t]}L_v\right)-g_t(\ell)\right)\vee 0\ \mu(\mathrm{d} t)\\&\quad+\mathbb{1}_{H^c}
		                  \int_{(T,\infty)}\left(g_t\left(\sup_{v\in (T,t]}L_v\right)-g_t(\ell)\right)\vee 0\ \mu(\mathrm{d} t)\\                  
		                  &\leq \mathbb{1}_{H^{(i)}_\ell}\int_{[T_\ell^{(i)},\infty)}\left(g_t\left(\sup_{v\in [T_\ell^{(i)},t]}L_v\right)-g_t(\ell)\right)\vee 0\, \mu(\mathrm{d} t) \\
		                  &\qquad +\mathbb{1}_{(H_\ell^{(i)})^c}\int_{(T_\ell^{(i)},\infty)}\left(g_t\left(\sup_{v\in (T_\ell^{(i)},t]}L_v\right)-g_t(\ell)\right)\vee 0\, \mu(\mathrm{d} t)
		               \end{align}
		               with equality in both steps for $\tau$ replaced by $\tau_\ell^{(i)}$.
		               For optimality of $\tau_{\ell}^{(i)}$, it thus suffices to show that we also have equality in \eqref{Main:0} for $\tau=\tau_\ell^{(i)}$. 
		               Recall that the inequality there is due to the application of Fatou's Lemma in \eqref{Main:2}. For $\tau=\tau_\ell^{(i)}$, the integrand is
		               bounded from below by $g(\ell)$ and from above by $g(\sup_{v\in [0,\cdot]}L_v)$, which justifies the application of dominated
		               convergence by assumptions on g and \eqref{sto:frame_gl_2}. This finishes our proof. 
            \end{proof}


\bibliographystyle{plainnat} \bibliography{Literature_ALL}

\begin{thebibliography}{12}
\providecommand{\natexlab}[1]{#1}
\providecommand{\url}[1]{\texttt{#1}}
\expandafter\ifx\csname urlstyle\endcsname\relax
  \providecommand{\doi}[1]{doi: #1}\else
  \providecommand{\doi}{doi: \begingroup \urlstyle{rm}\Url}\fi

\bibitem[{Bank} and {Besslich}(2018{\natexlab{a}})]{BB18}
P.~{Bank} and D.~{Besslich}.
\newblock {On a Stochastic Representation Theorem for {M}eyer-measurable
  Processes and its Applications in Stochastic Optimal Control and Optimal
  Stopping}.
\newblock \emph{ArXiv e-prints}, October 2018{\natexlab{a}}.

\bibitem[{Bank} and {Besslich}(2018{\natexlab{b}})]{BB18_3}
P.~{Bank} and D.~{Besslich}.
\newblock {Modelling information flow in stochastic optimal control: How
  {M}eyer-$\sigma$-fields settle the clash between exogenous and endogenous
  jumps}.
\newblock \emph{ArXiv e-prints}, October 2018{\natexlab{b}}.

\bibitem[Bank and El~Karoui(2004)]{BK04}
Peter Bank and Nicole El~Karoui.
\newblock A stochastic representation theorem with applications to optimization
  and obstacle problems.
\newblock \emph{Ann. Probab.}, 32\penalty0 (1B):\penalty0 1030--1067, 2004.
\newblock ISSN 0091-1798.
\newblock \doi{10.1214/aop/1079021471}.
\newblock URL \url{https://doi.org/10.1214/aop/1079021471}.

\bibitem[Bank and F\"{o}llmer(2003)]{BF03}
Peter Bank and Hans F\"{o}llmer.
\newblock American options, multi-armed bandits, and optimal consumption plans:
  a unifying view.
\newblock In \emph{Paris-{P}rinceton {L}ectures on {M}athematical {F}inance,
  2002}, volume 1814 of \emph{Lecture Notes in Math.}, pages 1--42. Springer,
  Berlin, 2003.
\newblock \doi{10.1007/978-3-540-44859-4_1}.
\newblock URL \url{https://doi.org/10.1007/978-3-540-44859-4_1}.

\bibitem[Besslich(2019)]{Bes19}
David Besslich.
\newblock \emph{Information flow in stochastic optimal control and a stochastic
  representation theorem for Meyer-measurable processes}.
\newblock PhD thesis, Technische Universit\"{a}t Berlin, 2019.

\bibitem[Bismut and Skalli(1977)]{BS77}
Jean-Michel Bismut and Bernard Skalli.
\newblock Temps d'arr\^{e}t optimal, th\'{e}orie g\'{e}n\'{e}rale des processus
  et processus de {M}arkov.
\newblock \emph{Z. Wahrscheinlichkeitstheorie und Verw. Gebiete}, 39\penalty0
  (4):\penalty0 301--313, 1977.
\newblock \doi{10.1007/BF01877497}.
\newblock URL \url{https://doi.org/10.1007/BF01877497}.

\bibitem[Dellacherie and Lenglart(1982)]{DL82}
C.~Dellacherie and E.~Lenglart.
\newblock Sur des probl\`emes de r\'{e}gularisation, de recollement et
  d'interpolation en th\'{e}orie des processus.
\newblock In \emph{Seminar on {P}robability, {XVI}}, volume 920 of
  \emph{Lecture Notes in Math.}, pages 298--313. Springer, Berlin, 1982.
\newblock \doi{10.1007/BFb0092793}.
\newblock URL \url{https://doi.org/10.1007/BFb0092793}.

\bibitem[Dellacherie and Meyer(1978)]{DM78}
Claude Dellacherie and Paul-Andr\'e Meyer.
\newblock \emph{Probabilities and potential}, volume~29 of \emph{North-Holland
  Mathematics Studies}.
\newblock North-Holland Publishing Co., Amsterdam-New York; North-Holland
  Publishing Co., Amsterdam-New York, 1978.
\newblock ISBN 0-7204-0701-X.

\bibitem[Dellacherie and Meyer(1982)]{DM82}
Claude Dellacherie and Paul-Andr\'e Meyer.
\newblock \emph{Probabilities and potential. {B}}, volume~72 of
  \emph{North-Holland Mathematics Studies}.
\newblock North-Holland Publishing Co., Amsterdam, 1982.
\newblock ISBN 0-444-86526-8.
\newblock Theory of martingales, Translated from the French by J. P. Wilson.

\bibitem[El~Karoui(1981)]{EK81}
N.~El~Karoui.
\newblock Les aspects probabilistes du contr\^ole stochastique.
\newblock In \emph{Ninth {S}aint {F}lour {P}robability {S}ummer {S}chool---1979
  ({S}aint {F}lour, 1979)}, volume 876 of \emph{Lecture Notes in Math.}, pages
  73--238. Springer, Berlin-New York, 1981.

\bibitem[Kiiski and Perkki\"o(2017)]{KP17}
Matti Kiiski and Ari-Pekka Perkki\"o.
\newblock Optional and predictable projections of normal integrands and
  convex-valued processes.
\newblock \emph{Set-Valued Var. Anal.}, 25\penalty0 (2):\penalty0 313--332,
  2017.
\newblock ISSN 1877-0533.
\newblock \doi{10.1007/s11228-016-0381-8}.
\newblock URL \url{http://dx.doi.org/10.1007/s11228-016-0381-8}.

\bibitem[Lenglart(1980)]{EL80}
E.~Lenglart.
\newblock \emph{Tribus de {M}eyer et th\'eorie des processus}, volume 784 of
  \emph{Lecture Notes in Math.}
\newblock Springer, Berlin, 1980.

\end{thebibliography}

\end{document}